\documentclass[11pt]{amsart}
\usepackage{amsmath, amsfonts, amssymb, amsthm}
\usepackage[norefs,nocites]{refcheck}
\usepackage{mathrsfs}
\usepackage[all]{xy}
\usepackage{array} 
\usepackage{mathtools}
\usepackage{faktor}
\usepackage[capitalize]{cleveref}

\makeatletter
\newcommand{\refcheckize}[1]{%
  \expandafter\let\csname @@\string#1\endcsname#1%
  \expandafter\DeclareRobustCommand\csname relax\string#1\endcsname[1]{%
    \csname @@\string#1\endcsname{##1}\@for\@temp:=##1\do{\wrtusdrf{\@temp}\wrtusdrf{{\@temp}}}}%
  \expandafter\let\expandafter#1\csname relax\string#1\endcsname
}
\newcommand{\refcheckizetwo}[1]{%
  \expandafter\let\csname @@\string#1\endcsname#1%
  \expandafter\DeclareRobustCommand\csname relax\string#1\endcsname[2]{%
    \csname @@\string#1\endcsname{##1}{##2}\wrtusdrf{##1}\wrtusdrf{{##1}}\wrtusdrf{##2}\wrtusdrf{{##2}}}%
  \expandafter\let\expandafter#1\csname relax\string#1\endcsname
}
\makeatother
\refcheckize{\cref}
\refcheckize{\Cref}
\refcheckizetwo{\crefrange}
\refcheckizetwo{\Crefrange}

\usepackage{verbatim,eucal}
\usepackage{xfrac}

\numberwithin{equation}{section}

\newtheorem{conjecture}[equation]{Conjecture} 
\newtheorem{theorem}[equation]{Theorem} 
\newtheorem{prop}[equation]{Proposition}
\newtheorem{lemma}[equation]{Lemma} 
\newtheorem{cor}[equation]{Corollary}
\newtheorem{example}[equation]{Example}
\newtheorem{remark}[equation]{Remark}
\newtheorem{definition}[equation]{Definition}

\newcommand{\td}{\tilde}
\newcommand{\qbin}[3]{ {{#1} \brack {#2}}_{#3}}

\newcommand{\Ipm}{\mathfrak{m}^{[p^m]}}
\newcommand{\Iqm}{\mathfrak{m}^{[q^m]}}
\newcommand{\Fp}{\mathbb{F}_p}
\newcommand{\FFp}{\mathbb{F}_p}
\newcommand{\FFq}{\mathbb{F}_q}
\newcommand{\om}{\omega}

\newcommand{\GLq}{{\GL}_n(\FF_q)}
\newcommand{\Glq}{{\GL}_n(\FF_q)}
\newcommand{\GLp}{{\GL}_n(\FF_p)}
\newcommand{\Glp}{{\GL}_n(\FF_p)}

\newcommand{\Cat}{{\rm{Cat}}}
\DeclareMathOperator{\Hilb}{Hilb}
\newcommand{\thilb}{\Hilb}

\DeclareMathOperator{\LM}{LM}

\DeclareMathOperator{\Ker}{Ker}

\DeclareMathOperator{\GL}{GL}
\DeclareMathOperator{\Gl}{GL}
\newcommand{\mfrac}[3][.85]{\scalebox{#1}{${\displaystyle\frac{#2}{#3}}$}}

\newcommand{\mbinom}[3][0.9]{\scalebox{#1}{$\dbinom{#2}{#3}$}}

\newcommand{\NN}{\mathbb N}
\newcommand{\FF}{\mathbb F}
\newcommand{\RR}{\mathbb R}

\renewcommand{\ker}{\mbox{\rm Ker\,}}

\newcommand{\ot}{\otimes}

\newcommand{\CC}{\mathbb{C}}  
\newcommand{\ZZ}{\mathbb{Z}}

\newcommand{\del}{\partial}

\begin{document}
\begin{abstract}
Lewis, Reiner, and Stanton
conjectured a Hilbert series
for a space of
invariants under an action
of finite general linear groups
using $(q,t)$-binomial coefficients. 
This work gives an analog
in positive
characteristic of theorems
relating
various Catalan numbers
to the representation theory
of rational Cherednik algebras.
They consider a finite
general linear group
as a reflection group acting
on the quotient
of a polynomial ring by iterated 
powers of the irrelevant ideal
under the Frobenius map.
 We
prove a variant of their
conjecture in the local case,
when the group acting fixes a
reflecting hyperplane.
\end{abstract}

\title[Invariants of polynomials mod Frobenius powers]
{Invariants of polynomials mod Frobenius powers}

\date{April 2, 2020}
\thanks{}

\author{C.\ Drescher}
\address{Department of Mathematics\\University of North Texas,
Denton, Texas 76203, USA}
\email{chelseadrescher@my.unt.edu}
\author{A.\ V.\ Shepler}
\address{Department of Mathematics, University of North Texas,
Denton, Texas 76203, USA}
\email{ashepler@unt.edu}
\thanks{The second author was partially supported by Simons grant 429539.}

\maketitle

\section{Introduction}

In 2017, Lewis, Reiner and Stanton~\cite{LRS} conjectured a combinatorial formula for the Hilbert series of a space of invariants under the action of the general linear group
$\GLq$ over 
a finite field $\FF_q$
in terms of $(q,t)$-binomial coefficients.
This formula provides an
analogue
for the $q$-Catalan and 
$q$-Fuss Catalan numbers
which connect Hilbert series
for certain invariant spaces
with the representation
theory of rational Cherednik algebras
for Coxeter 
and complex
reflection groups.
Results in the theory
of reflection groups often follow
from a local argument
after considering the subgroup
fixing one
reflecting hyperplane.
We prove here a version of the  conjecture in the local case.
We expect this local theory will extend to one 
for any modular reflection group, including $\Gl_n(\FF_q)$.

Lewis, Reiner, and Stanton
consider $\Glq$ acting on $V=(\FF_q)^n$
and
the polynomial ring 
$S=S(V^*)=\FF_q[x_1,\, \ldots, x_n]$
by transformation
of variables $x_1, \ldots, x_n$ in $V^*$. 
They consider the quotient of $S$
by the
$m$-th iterated Frobenius power of the irrelevant ideal,
$$\mathfrak{m}^{[q^m]}:= (x_1^{q^m}, \dots, x_n^{q^m})\, ,$$
which we call
the {\em Frobenius irrelevant ideal}.
Their conjecture gives the Hilbert series
for the $\Glq$-invariants in 
$\FF_q[x_1, \ldots, x_n]/(x_1^{q^m},
\ldots, x_n^{q^m})$ using
$(q,t)$-binomial
coefficients.

We consider subgroups
of reflections about a single hyperplane $H$ in $V$.
These groups are not cyclic in general,
in contrast to groups 
over fields of characteristic
zero.
We first take the case when $q$ is a prime $p$ and then generalize
some of our ideas
to arbitrary $q$.
We explicitly describe
the space of 
$G$-invariants in $S/\Ipm$
for any subgroup $G\subset
\Glp$ fixing a hyperplane $H$
in $V$ pointwise. We
give the Hilbert series
in terms of 
the dimension of 
the 
transvection root space.
We then describe the invariants
under the full pointwise
stabilizer $\Glq_H$ in $\Glq$
 of any hyperplane $H$
 for $q$ a prime power:
\begin{theorem}
\label{introthm}
For any hyperplane $H$ in $V=\FF_q^n$,
\begin{equation*}
    \begin{aligned}
     \Hilb \Big(  \big(\faktor{S}{\Iqm}
     \big)^{
     {\GLq}_H},\ t \Big)
\ =& \
 (
        [q^{m-1}]_{t^q}
        )^{n-1} 
        \qbin{m}{1}{{q,t}}
        + 
        t^{q^m -1}
        (
        [q^m]_{t}
        )^{n-1}
        \qbin{m}{0}{{q,t}}
\, .
\end{aligned}
\end{equation*}
\end{theorem}
\noindent
Recall
the $q$-integer
$[m]_q=1+q+q^2+\ldots+q^{m-1}$ and $(q,t)$-binomial coefficient 
(see~\cite{ReinerStanton})
\begin{equation*}
    {m \brack k}_{q,t}:= \prod_{i=0}^{k-1}\ \mfrac{1-t^{q^m-q^i}}{1-t^{q^k-q^i}}
    \ .
\end{equation*}
We  compare with
the Lewis, Reiner, Stanton conjecture in \cref{motivation}
and give this Hilbert series
in terms of $q$-Fuss Catalan
numbers.
The conjecture
implies
that the dimension
over $\FF_q$
of the ${\Glq}$-invariants in $S/\Iqm$
counts the number
of orbits in $(\FF_{q^m})^n$ under the action of
$\GLq$ and that
this dimension is
$\sum_{k=0}^{\min(n,m)}
{m\brack k}_q
$ (see~\cite[Section~7.1 and Theorem~6.16]{LRS}).
We prove an analogous
statement 
in \cref{q}:

\vspace{2ex}
\begin{cor}\label{introcor}
For any hyperplane
$H$ in $V=(\FFq)^n$, 
the number of orbits in $(\FF_{q^m})^n$ under the action of
$\Glq_H$ is
$$
\begin{aligned}
\dim_{\FFq}
\big(
\faktor{S}{\Iqm} 
        \big)^{{\Glq}_H}\ =  q^{(m-1)(n-1)}{ m \brack 1}_q + q^{m(n-1)}{m \brack 0}_q\, .
\end{aligned}
$$
\end{cor}

\begin{example}{\em 
 Consider $G$ acting on $V=(\mathbb{F}_5)^3$ with $\dim_{\FF_5}(
 \text{RootSpace}(G)\cap H)=2$. 
 Then $G$ is generated by two transvections and possibly a diagonalizable reflection.  We may assume (after a change-of-basis) that 
\begin{equation*}
    G=\left<\Big(\begin{smallmatrix} 1 & 0 & 0 \\
    0 & 1 & 0\\
    0 & 0 & \omega \end{smallmatrix}\Big), \Big(\begin{smallmatrix} 1 & 0 & 1 \\
    0 & 1 & 0 \\
    0 & 0 & 1
    \end{smallmatrix}\Big), \Big(\begin{smallmatrix} 1 & 0 & 0 \\
    0 & 1 & 1 \\
    0 & 0 & 1
    \end{smallmatrix}\Big) \right>
\end{equation*}
for some primitive $e$-th root-of-unity $\omega$ in $\FF_5$. The $m$-th iterated irrelevant ideal
in $\FF_5[x_1, x_2, x_3]$ 
is
$(x_1^{5^m}, x_2^{5^m},x_3^{5^m})$ for $m \geq 1$.
We will see in ~\cref{section: hilbert series maximal transvection} that
\begin{equation*}
    \begin{aligned}
        &\Hilb\Big(\big(\faktor{S}{\mathfrak{m}^{[5^m]}}\big)^{G}, \ t\Big)= \frac{(1-t^{5^m})^2}{(1-t^5)^2(1-t^e)}\Big(1-t^{5^m-1}+t^{5^m-1}(1-t^e)\big(\mfrac{1-t^5}{1-t}\big)^2\Big) \ .
    \end{aligned}
\end{equation*}}
\end{example}

\subsection*{Outline}
In ~\cref{motivation}, we give  motivation from the theory of rational Catalan combinatorics, which relates
rational Cherednik algebras with
various kinds of Catalan numbers.
We recall some facts on modular reflection groups
in \cref{choosebasis}.
In \Cref{section: description of A,section: hilbert series A,section: direct sum decomp,section: hilbert series maximal transvection},  we mainly consider a 
subgroup $G$ of $\GLp$ fixing a hyperplane $H$ with maximal transvection root space; more general results in \Cref{section: direct sum decomp,section:final result} will follow from this special case. 
We give a Groebner basis for $S^G \cap \Ipm$ 
in \cref{section: description of A} and compute the Hilbert series for $S^G /(S^G\cap\Ipm)$ in \cref{section: hilbert series A}.
In \cref{section: direct sum decomp}, we decompose 
$(S/\Ipm)^G$ as 
the direct sum of $S^G/(S^G\cap\Ipm)$
and a complement.
We give the Hilbert series for the $G$-invariants in
$S / \Ipm$ 
when $G$ has maximal
root space
in ~\cref{section: hilbert series maximal transvection} 
and for general groups
fixing a  hyperplane over 
$\FF_p$ in \cref{section:final result}.
We consider
the full pointwise stabilizer of a hyperplane
in $\GLq$ in ~\cref{q}: We
establish \cref{introthm} and show the Hilbert series
counts orbits.
We give a bound
on the Hilbert series for $\GLq$
in the conjecture of Lewis, Reiner, and Stanton
in \cref{section: LRS tie in}. 
Lastly, in \cref{section: 2dim}, we give a resolution
directly for $S^G\cap \Ipm$ in the $2$-dimensional case.

\section{Motivation}
\label{motivation}

We recall some motivation
for studying the invariants of
$S/\Iqm$ from 
the theory of Catalan combinatorics
for
Coxeter and complex reflection groups;
see  
Armstrong, Reiner, and Rhoades~\cite{ArmstrongReinerRhoades};
Berest, Etingof and Ginzburg~\cite{BEG};
Bessis and Reiner~\cite{BessisReiner};
Gordon~\cite{Gordon};
Gordon and Griffeth~\cite{GordonGriffeth};
Krattenthaler and M\"uller~\cite{Krattenthaler};
and Stump~\cite{Stump}.

\subsection*{Graded Parking Spaces
and Rational Cherednik Algebras}
The {parking space}
of an irreducible Weyl group
gives an irreducible representation
of the associated rational Cherednik algebra.
The $q$-Catalan number for the group
records the Hilbert series
for the invariants in this space
in terms of the degrees
$d_1, \ldots, d_n$
and Coxeter number $h$ 
(see~\cite{GordonGriffeth})
of the reflection group.
More generally, for an irreducible
Coxeter group $W$ acting on
$V=\CC^n$,
the graded parking
space representation
(see~\cite{ArmstrongReinerRhoades})
is isomorphic to 
$S/(\theta_1, \ldots, \theta_n)$
for some homogeneous polynomials
$\theta_1, \ldots, \theta_n$
in $S$ of degree $h+1$
with
$\CC\text{-span}
\{\theta_1, \ldots, \theta_n\}$
isomorphic
to the reflection representation
$V^*$.
The $W$-invariants in the parking space
has 
Hilbert series given by the $q$-Catalan number for $W$:
$$
\Hilb\Big(
\big(
\faktor{S}{(\theta_1,
\ldots, \theta_n)\,}\big)^W 
,\ q\Big)
=\Cat(W,q)
=\prod_{i=1}^n
\mfrac{
1-q^{h+d_i} }
{1-q^{d_i} }\, .
$$
For a complex
reflection group $W$,
Gordon and Griffeth~\cite{GordonGriffeth} 
connect the representation theory
of the associated rational Cherednik algebra
to 
the $m$-th $q$-Fuss
Catalan numbers,
$$
\Cat^{(m)}(W,q)
=
\prod_{i=1}^n \mfrac{[d_i+mh]_q}
{[d_i]_q}
\ =\
\prod_{i=1}^n \mfrac{1-q^{d_i+mh}}{1-q^{d_i}}\ ,
$$
giving
the Hilbert series
of $W$-invariants in
a space $
S/(\td\theta_1,
\ldots, \td\theta_n)
$
with $\deg(\td\theta_i)=mh+1$.

\subsection*{Lewis, Reiner, 
and Stanton Conjecture}
The ideal $(\theta_1, \ldots, \theta_n)$ takes
a particularly
nice form for some Coxeter groups
with $\theta_i=x_i^{h+1}$;
the graded parking space in
this case is just
$\CC[x_1, ..., x_n]/(x_1^{h+1},
\ldots, x_n^{h+1})$.
Lewis, Reiner, and 
Stanton~\cite{LRS} ask what ideal can play the role of $(\theta_1, \ldots, \theta_n)$
for the modular
reflection group
$\text{GL}_n(\FF_q)$.
They consider the ideal
$(\theta_1,\ldots,\theta_n)=(x_1^{q^m},\ldots,x_n^{q^m})=\Iqm$
for $m \geq 0$ 
since $\theta_1,\ldots, \theta_n$
span a $\text{GL}_n(\FF_q)$-stable subspace
over $\FF_q$ with the map $x_i \mapsto x_i^{q^m}$ defining
a $\text{GL}_n(\FF_q)$-equivariant isomorphism (see~\cite{LRS}). 
The quotient
$S/\Iqm$
is $(q^m)^n$-dimensional,
and 
Lewis, Reiner, and Stanton
give a
conjecture for the Hilbert series of its $\text{GL}_n(\FF_q)$-fixed subspace:

\begin{conjecture}[\cite{LRS}]
\label{conjecture}
The space of ${\GLq}$-invariants in $\faktor{S}{\mathfrak{m}^{[q^m]}}$ 
has Hilbert series
\begin{equation*}
\begin{aligned}
\Hilb\Big( \big(\faktor{S}{\mathfrak{m}^{[q^m]}} 
        \big)^{{\GLq}},\ 
        t \Big)
   =&\ 
   \sum_{k=0}^{
   \min(n,m)}t^{(n-k)(q^m-q^k)} 
   {m \brack k}_{q,t}\, 
   \\
   =&\
    \sum_{k=0}^{
   \min(n,m)}
   t^{(n-k)(q^m-q^k)}
    \mfrac{ \Hilb( S^{P_k}, t ) }
        {
        \Hilb( S^{\GL_m(\FF_{q})}, t)}
        \, 
\end{aligned}
\end{equation*}
for $P_k$ the maximal parabolic subgroup of ${\GL_m(\FF_q)}$ stabilizing
any $\FF_q$-subspace
of $(\FF_q)^m$ isomorphic
to $(\FF_q)^k$.
\end{conjecture}

Compare with our \cref{introthm},
which is equivalent to the statement that
\begin{equation*}
\begin{aligned}
\Hilb\Big( \big(\faktor{S}{\mathfrak{m}^{[q^m]}} 
        \big)^{{\GLq}_H},\ 
        t \Big)
\ = \
\mfrac{
\Hilb(S^{\text{GL}_n(\FF_{q})_H}, t)}
{\Hilb(S^{\text{GL}_n(\FF_{q^m})_H}, t)}
& \ +\
 \mfrac{         (t^{q^m-1}-t^{q^m})}
         { (1-t^{q^m-1}) }
\mfrac{\Hilb(S,\ t)}
{\Hilb(S^{\text{GL}_n
(\FF_{q^m})_H}, t)}
\, .
\end{aligned}
\end{equation*}
\subsection*{A curious reformulation}
We mention a version
of \cref{introthm}
in terms of $q$-Fuss
Catalan numbers
that connects with
\cref{conjecture};
we wonder if a version of this
reformulation holds for other
reflection groups.
For modular reflection groups,
the above definition of Coxeter number
does not always give an integer,
so
we use an alternate
definition that agrees with the traditional
one over $\RR$ or $\CC$.
For
any reflection group $G$
acting on $V$
with a polynomial ring
of invariants $S^G=
\FF[f_1, \ldots, f_n]$,
define the
$$
\text{{\em Coxeter number} of $G$}
:=
\mfrac{\deg J + \deg Q}{n}\, 
$$
for $J=\det \{ \del f_i/\del x_j \}
_{i,j=1, \ldots, n}$
in $S$, the determinant of
the Jacobian derivative matrix,
and $Q=\prod_{H\in \mathcal{A}} 
l_H$, the polynomial
in $S$ defining the arrangement
$\mathcal{A}$ of reflecting 
hyperplanes for $G$.
Note that $\deg J$ is {\em not}
the number of reflections in $G$ in general.

For any hyperplane $H$ in $V=\FF_q^n$
and $G=\Glq_H$,
\cref{introthm} implies that
\begin{equation}
\label{reformulation}
    \begin{aligned}
     \Hilb \Big(  \big(\faktor{S}{\Iqm}
     \big)^{G}
     ,\ t \Big)
\ =& \
\sum_{k=0,1}
t^{(n-\dim G_k)(q^m-q^k)}
\ 
\Cat^{(c_k)}
(G_k, t)
\end{aligned}
\end{equation}
where
$c_k=(q^m-q^k)/h_k$
and
$G_k= 
\big(\text{Stab}_{G}(V_k)\big)|_{V_k}
$
(setwise stabilizer)
with Coxeter number $h_k$
for
$V_0=H$ and $V_1=V$.
Here, $G_0$ is 
the identity subgroup of $\Gl_{n-1}(\FF_q)$
regarded as the direct sum of
trivial reflection groups with degrees $1, \ldots, 1$
and Coxeter number 
$1$
while $G_1=G$ has Coxeter number $q-1$.
Each 
Fuss parameter
$c_k$ lies in $\NN$
although $G_k$ is reducible.

Although reformulation \cref{reformulation} is
somewhat artificial, it
agrees with a version
of the Lewis, Reiner,
and Stanton conjecture
if we allow for non-integer
Fuss parameters.
For $G=\Glq$,
Conjecture~\ref{conjecture}
implies that
\begin{equation*}
    \begin{aligned}
     \Hilb \Big(  \big(\faktor{S}{\Iqm}
     \big)^{
     G},\ t \Big)
\ =& \
\sum_{k=0}^{\min\{n,m\}}
t^{(n-\dim G_k)(q^m-q^k)}
\ 
\Cat^{c_k}
(G_k, t)
\, 
\end{aligned}
\end{equation*}
where again
$c_k=(q^m-q^k)/h_k$ 
and
$G_k= 
\big(\text{Stab}_{G}(V_k)\big)|_{V_k}
=\Gl_k(\FF_q)$
with Coxeter number $h_k=q^k-1$
for $V_k=(\FF_q)^k\subset (\FF_q)^n$.
Here, at least 
the groups $G_k$ are irreducible.

\section{Reflection groups and transvections}
\label{choosebasis}
Recall that a {\em reflection} on $V=\FF^n$ for any field $\FF$
is a transformation $s$ in $\GL(V)$ whose fixed point space is a hyperplane $H$ in $V$. A {\em reflection group} is a subgroup of $\GL(V)$ generated by reflections;
we assume all reflection groups are finite.
Suppose $G$ is a reflection group
fixing a hyperplane $H$ in $V$
and choose some 
linear form $l$ in $V^*$ defining $H$,
i.e., with $\ker l=H$.
Every $g$ in $G$ 
defines a {\em root vector} 
$\alpha_g$ in $V$ satisfying
$$
g(v) = v+ l(v) \alpha_g
\quad\text{ for all } v\text{ in } V\, .
$$
We denote the collection of all
root vectors by
$\text{RootSpace}(G)$.
In the {\em nonmodular setting},
when the characteristic $p$ is relatively prime to $|G|$, 
the group $G$ is cyclic.
In this case,
every group element is semisimple,
and
one can choose a $G$-invariant inner product
so that any root vector for $H$ is perpendicular to $H$.
In the 
{\em modular setting},
when $p=\text{char}(\FF)$ divides
$|G|$,
the root vector of a reflection $g$ 
may lie {\em in} $H$ itself;
this occurs exactly when $g$ is not 
semisimple.
Such reflections are called {\em transvections}
and they have order 
$p=\text{char}(\FF)$.

The transvections in $G$
form a normal subgroup 
$K$, the kernel
of the determinant character
$\det: G \rightarrow \FF^{\times}$.
The group
$G$ is generated by $K$
and some semisimple
element $g_n$ of maximal order 
$e=|G/K|$,
and $G$ 
is isomorphic
to the semi-direct product of
$K$ and the cyclic subgroup 
$\langle g_n \rangle $ of
semisimple reflections:
\begin{equation*}
    G\cong
    K\,    
    \rtimes
    \ZZ/ e\ZZ
    \, .
\end{equation*}

Now assume $\FF=\FF_p$
for a prime $p$.
The
corresponding {\em transvection root space}
$\text{RootSpace}(G)\cap H$ is an 
$\FF_p$-vector
space (see~\cite{HartmannShepler}),
and its dimension,
$$
\ell=\dim_{\FF_p}\big(\text{RootSpace}(G)\cap H\big)
\, ,
$$
is the minimal number of
transvections needed to generate $G$:
there are transvections $g_1, \ldots, g_\ell$ with
$
G = \langle 
g_1, \ldots, g_{\ell}, g_n\rangle\, 
$
and
$|G|=e\cdot p^{\ell}$.

 \subsection*{After conjugation}
We may choose a basis $v_1, \dots, v_n$ of $V$ with dual basis $x_1, \dots, x_n$ of $V^*$ so that $v_1, \dots, v_{n-1}$ span the hyperplane $H=\ker(x_n)$,
$g_n$ fixes
$x_1, \ldots, x_{n-1}$,
and $g_n(x_n) = \om^{-1} x_n$
for $\om$ a primitive $e$-th root-of-unity
in $\FF_p$.
We furthermore
refine the basis
so that each transvection
$g_k$ fixes $x_1, \ldots, x_{k-1}, x_{k
+1}, \ldots, x_n$
and $g_k(x_k) = x_k - x_n$:
\begin{equation*}
g_n = \left(\begin{smallmatrix}
1   & \hdots & 0 & 0\\
\vdots & \ddots & & \vdots\\
0   & \hdots & 1 & 0\\
0   & \hdots & 0 & \omega \\
 \end{smallmatrix} \right)
 \quad\text{ and,}
 \text{ for } 1 \leq k \leq \ell,\quad
g_{k}:= \left(\,\begin{smallmatrix}
1 & \cdots & & &  &\cdots  & 0\\
 \vdots & \ddots &  &  &  &  & \vdots
\rule[-.5ex]{0ex}{1ex}\\
0 & \cdots & 1 & 0 & 0 & \cdots  & 0\\
0  &\cdots & 0 & 1 & 0 & \cdots  & 1
&  & \ \ 
\leftarrow k^{\text{th}} \text{ row }
\rule[-.25ex]{0ex}{1ex}
\\
0 &\cdots & 0 & 0 & 1 & \cdots  & 0
\\
\vdots &  &  &   & &
\ddots &\vdots &  
\rule[-.5ex]{0ex}{1ex}\\
0 & \cdots &  &  & &\cdots  & \ 1 \  
 \end{smallmatrix}
 \hspace{-12ex} \right)
\hspace{9ex}
 .
\end{equation*}

\begin{example}{\em 
When $n=3$, $p=5$,
and $\ell=1$, $G$ 
acting on $V=(\FF_5)^3$
is generated by one 
transvection and
possibly an additional
semisimple reflection.
We may assume (after a change-of-basis)
that
for some $e$-th root-of-unity $\omega$ in $\FF_5$
\begin{equation*}
    G=\left< 
     g_3=\Big(\begin{smallmatrix} 1 & 0 & 0 \\
    0 & 1 & 0\\
    0 & 0 & \omega \end{smallmatrix}\Big)   
    , \
    g_1=\Big(\begin{smallmatrix} 1 & 0 & 1 \\
    0 & 1 & 0 \\
    0 & 0 & 1
    \end{smallmatrix}\Big)
    \right>\, .
\end{equation*}
}
\end{example}

\subsection*{Basic Invariants}
The ring of invariant polynomials $S^G$
is itself a polynomial ring,
$S^{G}=\mathbb{F}_p[ f_1, \dots , f_{n}]
$
with homogeneous generators
$$
f_1=x_1^p-x_1 x_n^{p-1},
\ldots,
f_{\ell} = x_{\ell}^p - x_{\ell} x_n^{p-1}
,\ \
f_{\ell+1} = x_{\ell+1},
\ldots,
f_{n-1} = x_{n-1},
\ \ 
f_n=x_n^e \, 
$$ 
and
$$
\Hilb(S^G, t) = 
\mfrac{1}{
(1-t^p)^{\ell}\
(1-t)^{n-\ell-1}\
(1-t^{e})}\, .
$$

\begin{example}{\em 
For 
$\rule[-2ex]{0ex}{1ex}
    G=\left<\Big(\begin{smallmatrix} 1 & 0 & 0 \\
    0 & 1 & 0\\
    0 & 0 & \omega \end{smallmatrix}\Big), \Big(\begin{smallmatrix} 1 & 0 & 1 \\
    0 & 1 & 0 \\
    0 & 0 & 1
    \end{smallmatrix}\Big), \Big(\begin{smallmatrix} 1 & 0 & 0 \\
    0 & 1 & 1 \\
    0 & 0 & 1
    \end{smallmatrix}\Big)  \right> \subset \GL_3(\FF_5),
$ the ring
$S^G$ is generated  by
$f_1=x_1^5-x_1 x_3^{4}$,
$f_2=x_2^5-x_2 x_3^{4}$
and $f_3=x_3^e$
as an $\FF_5$-algebra
for $e=\text{order}(\om)$.
The Hilbert series of $S^G$ is
$$
\Hilb(\FF_5[x_1, x_2, x_3]^G, t)  
= \mfrac{1}{
(1-t^5)^2\
(1-t^{e})}\, .
$$
}
\end{example}


\section{Describing the invariants in the Frobenius irrelevant ideal}
\label{section: description of A}

We begin by finding the invariants in the Frobenius irrelevant ideal itself, describing
$S^{G} \cap\, \Ipm$.
Throughout this section,
we assume $G$ is a 
subgroup of $\Glp$ fixing a hyperplane $H$ of $V=\FFp$.
The general case will follow
from the special case when $G$ has
maximal transvection root space, so
we assume $\ell=n-1$.
Without loss of generality, we may
take a basis $x_1, \ldots, x_n$
for $V^*$ so that $H=\Ker x_n$
and
$G$ acts as in \cref{choosebasis}.

\subsection*{Monomial orderings}
We consider $S$ as a graded
ring with respect to the usual
polynomial degree with $\deg x_i=1$
for all $i$.
The Frobenius irrelevant ideal 
$\mathfrak{m}^{[p^m]}$ is then a homogeneous ideal
giving a graded quotient $S/ \mathfrak{m}^{[p^m]}$.
We use compatible monomial orderings on the two polynomial rings $S$ and $S^{G}$.  On $S=\FFp[x_1, \ldots, x_n]$, we take 
the graded lexicographical ordering with
$x_1 > x_2 > \cdots > x_n$.
On $S^{G}=\FFp[f_1,\ldots, f_n]$, we take the inherited
graded lexicographical ordering
with
$\deg(f_n)=e < p$,  $\deg(f_i)=p$ for $i < n$,
and
$f_1> f_2 > \cdots > f_n$.
Then for any polynomials 
$f$ and $f'$ in $S^G$,
$f<f'$ in the monomial ordering on $S^G$
if and only if $f<f'$ in the 
monomial ordering on $S$.
We use the notation $\textrm{LM}_{S}(f)$ and $\LM_{S^G}(f)$ for the leading monomials of a polynomial $f$ with respect to the ordering on $S$ and $S^G$, respectively. Then
\begin{equation}\label{leading monom}
    \textrm{LM}_{S}\big(  
\LM_{S^G}(f)
\big)\, = \LM_S(f) .
\end{equation}
We will frequently use the fact 
that for any nonnegative
exponents $a_i$ and $i<n$,
\begin{equation}
\label{modtheideal}
    \begin{aligned}
    &f_i\, x_1^{a_1}\dots x_{n-1}^{a_{n-1}}
    \,x_n^{p^m-1}
   \  \equiv \
    x_1^{a_1}\dots x_{i-1}^{a_{i-1}} x_i^{p+a_i}x_{i+1}^{a_{i+1}}\dots x_{n-1}^{a_{n-1}}
    \,x_n^{p^m-1}
    \quad\text{ mod } \Ipm\, . 
    \end{aligned}
\end{equation} 
\subsection*{Generators
for  invariants
in the Frobenius irrelevant ideal}

We will show that the following polynomials give a Groebner basis for $S^{G} \cap \Ipm$.

\begin{definition}
\label{definition:invariants}
Define polynomials in 
$S^G=\FF_p[f_1, \ldots, f_n]$ 
for $1 \leq a \leq b < n$
by
\begin{equation*}
\begin{aligned}
    h_0 &= f_n^{1+e^{-1}(p^m-1)}\, ,  
    &h_{1,a} &= \ \sum_{k=0}^{m-1}f_n^{1+e^{-1}(p^m-p^{m-k})}\ f_a^{p^{m-k-1}}\, ,
    \textrm{ and } h_{2,a,b} = f_a^{p^{m-1}}f_b^{p^{m-1}}. 
\end{aligned}
\end{equation*}
\end{definition}

\begin{example}{\em 
For our archetype example
$
    G=\left<\Big(\begin{smallmatrix} 1 & 0 & 0 \\
    0 & 1 & 0\\
    0 & 0 & \omega \end{smallmatrix}\Big), \Big(\begin{smallmatrix} 1 & 0 & 1 \\
    0 & 1 & 0 \\
    0 & 0 & 1
    \end{smallmatrix}\Big), \Big(\begin{smallmatrix} 1 & 0 & 0 \\
    0 & 1 & 1 \\
    0 & 0 & 1
    \end{smallmatrix}\Big) \right>
\subset
\Gl_3(\FF_5)$,
\begin{equation*}
   \begin{aligned}
    &h_0 = f_3^{1+e^{-1}(5^m-1)}\, , \\ 
    &h_{1,1} = \ \sum_{k=0}^{m-1}f_3^{1+e^{-1}(5^m-5^{m-k})}\ f_1^{5^{m-k-1}}\, , \quad h_{1,2} = \ \sum_{k=0}^{m-1}f_3^{1+e^{-1}(5^m-5^{m-k})}\ f_2^{5^{m-k-1}}\, , \\
     &h_{2,1,1} = f_1^{2(5^{m-1})}\, , \quad h_{2,1,2}=f_1^{5^{m-1}}f_2^{5^{m-1}}\, , \quad \textrm{ and } h_{2,2,2}=f_2^{2(5^{m-1})}. 
\end{aligned}
\end{equation*}
}
\end{example}
The next lemma verifies
that these polynomials lie in the Frobenius irrelevant ideal.

\begin{lemma}\label{lemma:definition of invariants in frobenius ideal}
For $G$ with maximal transvection root space,
the polynomials
$h_0, \, h_{1,a}, \, h_{2,a,b}$
for $1 \leq a \leq b < n$
lie in $S^G\cap\Ipm$. 
\end{lemma}
\begin{proof}
Straight-forward computation confirms that
\begin{equation*}
    \begin{aligned}
        &h_0 = x_n^{p^m+e-1},
        \quad
        h_{1,a} = x_a^{p^m}x_n^e-x_ax_n^{p^m+e-1}, \quad \text{and}\\
        &h_{2,a,b} 
        = 
        x_a^{p^m}x_b^{p^m} 
        \! -
        x_a^{p^{m-1}}x_b^{p^m}x_n^{(p-1)p^{m-1}} 
        \! -
        x_a^{p^{m}}x_b^{p^{m-1}}x_n^{(p-1)p^{m-1}}
        \!+ x_a^{p^{m-1}}x_b^{p^{m-1}}x_n^{2(p-1)p^{m-1}}\, .
            \end{aligned}
\end{equation*}
\end{proof}
The next key lemma describes elements
of $S^G\cap\, \Ipm$; it relies
on an inductive argument
using Lucas' Theorem 
on binomial
coefficients (see~\cite{Lucas} or~\cite[Exercise~1.6(a)]{Stanley}).

\begin{lemma}\label{lemma: leading monomial division}
If $G$ has maximal transvection root space and $f \in S^{G}\cap \Ipm$ 
is homogeneous
in $f_1, \ldots, f_{n}$, then 
$\LM_{S^G}(f)$ is divisible by
$f_n$ 
or some $h_{2,a,b}$
with $1\leq a\leq b <n$.
\end{lemma}
\begin{proof}
Suppose no $h_{2,a,b}$ divides $\LM_{S^G}(f)$ nor $f_n$. 
Then
\begin{equation*}
    \LM_{S^G}(f)=f_1^{c_1}f_2^{c_2}\cdots f_{n-1}^{c_{n-1}}
\end{equation*}
for some $c_i<2p^{m-1}$
(as no $h_{2,a,a}$ divides)
with all but possibly one exponent
satisfying $c_i< p^{m-1}$
(as no $h_{2,a,b}$ divides
for $a\neq b$).
Observe  first
that not all $c_i<p^{m-1}$. 
Otherwise, by the binomial theorem and \cref{leading monom},
\begin{equation*}
   \LM_S(f) = \LM_{S}(\LM_{S^G}(f))= x_1^{c_1p}x_2^{c_2p}\cdots x_{n-1}^{c_{n-1}p}
\end{equation*}
would not lie in 
the monomial ideal $\Ipm$,
contradicting the fact that $f$ does.
Hence there is a unique index $j$ 
with 
$p^{m-1}\leq c_j < 2p^{m-1}$. Without loss of generality, 
say $j=1$, so that
$p^{m-1}\leq c_1 < 2p^{m-1}$ and $c_i<p^{m-1}$ for $1<i<n$. 
Define $h$ by
\begin{equation*}
    h=f\cdot f_1^{2p^{m-1}-c_1-1}f_2^{p^{m-1}-c_2-1}f_3^{p^{m-1}-c_3-1}\cdots f_{n-1}^{p^{m-1}-c_{n-1}-1}.
\end{equation*}
We will produce a monomial 
$$
x_{\alpha} =
x_1^{p^m-p+1}\,
x_2^{p^m-p}x_3^{p^m-p}
\cdots x_{n-1}^{p^m-p}
\,
x_n^{p^m-1}
$$ 
of $h$ in the variables
$x_1, \ldots, x_n$ 
which does not lie in $\Ipm$.
This will imply that $h$ itself
does not lie in $\Ipm$, contradicting the fact that
$h$ is a multiple of $f$.

To this end, set $L=\LM_{S^G}(h)$,
so that, by construction, 
\begin{equation*}
    L=\LM_{S^G}(h)= f_1^{2p^{m-1}-1}f_2^{p^{m-1}-1}\cdots f_{n-1}^{p^{m-1}-1}\, .
\end{equation*}
We write $L$ 
as a polynomial in the variables
$x_1, \ldots, x_n$
using the binomial theorem. 
Direct
calculation in $S/\Ipm$  confirms that
\begin{equation*}
    L + \Ipm = 
    \pm \ x_\alpha + \Ipm
\end{equation*}
as
Lucas' Theorem 
implies that
\begin{equation*}
    \binom{2p^{m-1}-1}{\sum_{i=0}^{m-1}p^i} 
    = 
    \begin{cases} 1 & 
    \text{for } m=1,2,\\
     \prod_{i=0}^{m-2} \binom{p-1}{1} = (-1)^{m-1} & 
    \text{for } m>2\, . \end{cases}
\end{equation*}
Thus, the monomial
$x_{\alpha}$
appears with nonzero
coefficient in $L$
and does not lie in $\Ipm$.

We now argue 
that $x_\alpha$ appears with nonzero
coefficient in $h$ itself
(i.e., does not cancel
with other terms).
Consider the coefficient $
c_\alpha(M)$ of $x_\alpha$ 
in some other monomial
$$
M = f_1^{c_1'} f_2^{c_2'}
\cdots f_n^{c_n'}\ \ 
<\ L
$$
of $h$ after expanding $M$
in the variables $x_1, \ldots, x_n$ and
suppose 
$c_\alpha(M)\neq 0$.

We first establish that 
$M$ has smaller
degree in $f_1$ than $L$ but larger 
degree in $f_n$. 
Indeed, note 
that $p^{m-1} \leq c_1'$,
else $\deg_{x_1}(M) < \deg_{x_1}(x_\alpha)$ and $c_\alpha(M)=0$. Now fix $1<i<n$ and consider $c_i'$. Note that
$c_i' \geq p^{m-1}-1$ else $c_\alpha(M)=0$ as $f_k\in \FFp[x_k,x_n]$ for
all $k$. And 
$c_i' \leq p^{m-1}-1$ else
$h_{2,1,i}$ divides $M$ and $c_\alpha(M)=0$ as $M \in \Ipm$. Thus $c_i'=p^{m-1}-1$ and 
$\deg_{f_i}(M) = 
\deg_{f_i}(L)$
for $1<i<n$.
But $\deg_S M = 
\deg_S L$,
with $M < L$.  
Thus $M$ has smaller
degree in $f_1$ but larger 
degree in $f_n$ than $L$, i.e.,
\begin{itemize}
    \item $p^{m-1} \leq c_1' <2p^{m-1}-1$, and
    \item $c_i'=p^{m-1}-1$ for $1<i<n$, and
    \item $c_n'>0$.
\end{itemize}

We assume $m \geq 2$
since if $m=1$, then 
$c_1'=0$ and $\deg_{x_1}(M)=0$,
forcing $c_\alpha(M)=0$.

We examine the contribution to $M$ from $f_1$.
Set $d=c_1'$.
Then as $c_\alpha\neq 0$
and
\begin{equation*}
    f_1^d=(x_1^p-x_1x_n^{p-1})^d
    =
    \sum_{i=0}^d\mbinom{d}{i}\
    x_1^{dp-(p-1)i}x_n^{(p-1)i}\, ,
\end{equation*}
 there is some index $i$ with
 $\binom{d}{i} \neq 0$ and
$    dp-(p-1)i=p^m-p+1 $.
Hence $i\equiv 1\mod p$.
Since $d<2p^{m-1}-1$ by assumption, 
\begin{equation}
\label{dInTermsOfa}
d=
p^{m-1}+(p-1)a
\quad\text{and}\quad
i = 1+pa
\quad\text{for some}\quad
0 \leq a < \sum_{k=0}^{m-2}p^k\, .
\end{equation}

We show instead that $\sum_{k=0}^{m-2}p^k \leq a$  by considering the base $p$ expansions
of $a$ and $d$:
$$
    a=\sum_{k=0}^{m-2} a_k\, p^k
    \quad\text{ and }\quad
 d = \sum_{k=0}^{m-1}d_k\, p^k
 \quad\text{
for some }\ \ 
   0 \leq a_k,\, d_k <p. 
$$   
We compare the base $p$
coefficients $d_k$ and $a_k$
using the key point that 
$\binom{d}{i}$ is nonzero:
Lucas' Theorem
implies that
\begin{equation*}
    0\neq 
    \mbinom{d}{i} = \mbinom{d_0}{1}\prod_{k=1}^{m-1}
    \mbinom{d_k}
    {a_{k-1}}
\quad    \text{ as }
    i =  1+\sum_{k=1}^{m-1}
    a_{k-1}\, p^{k}\, ;
\end{equation*}
since no factor in the product vanishes,
we conclude that $d_0 \geq 1$ and each $a_{k-1}\leq d_k$. 
 \cref{dInTermsOfa} then provides direct
 comparison of $d_k$ and $a_k$,
\begin{equation}\label{d equation}
 \sum_{k=0}^{m-1}d_k\, p^k
 \ =\
    d\ =\ p^{m-1}-a_0 +\sum_{k=1}^{m-2}(a_{k-1}-a_k)p^{k+1} + a_{m-2}p^{m-1}.
\end{equation}
We now regroup base $p$ as needed
and show inductively that
$0<a_0\leq a_1\leq\ldots\leq a_{m-2}$.

We first consider  $a_0$. Since 
$1\leq d_0$, ~\cref{d equation} implies that $d_0=p-a_0$ and  $a_0 \neq 0$. For $m>2$, next observe that $a_0 \leq a_1$
since $a_0\leq d_1$
and~\cref{d equation} implies that
\begin{equation*}
    d_1=p+a_0-a_1-1 \textrm{ for } a_0 \leq a_1 \quad \textrm{whereas} \quad
    d_1=a_0-a_1-1 \textrm{ for } a_1 < a_0.
\end{equation*}
Similarly, $a_1 \leq a_2$ since $a_1 \leq d_2$ and ~\cref{d equation} implies that
\begin{equation*}
    d_2= p+a_1-a_2-1 \textrm{ for } a_1 \leq a_2 \quad \textrm{whereas} \quad d_2=a_1-a_2-1 \textrm{ for } a_2<a_1.
\end{equation*}
We iterate this argument and conclude
that 
$0<a_0\leq a_1\leq\ldots\leq a_{m-2}$.
But this contradicts ~\cref{dInTermsOfa}, so  $\binom{d}{i}=0$
and thus $c_\alpha(M)=0$. 
\end{proof}

We now show that the collection 
of
$h_0$, $h_{1,a}$, $h_{2,a,b}$
is a Groebner basis.
\begin{prop}\label{groebner basis}
If $G$ has maximal transvection root space,
then the ideal $S^G \cap \Ipm$ 
of $S^G$ has as 
a Groebner basis 
$
\mathscr{G} = \{h_0,\, h_{1,a},\, h_{2,a,b}:\,
1 \leq a \leq b < n\}
\, .
$
\end{prop}

\begin{proof}
Suppose $f$ 
in $S^G \cap \Ipm$ is homogeneous in the variables $f_1, \dots, f_n$.
Say neither
$h_0=\LM_{S^G}(h_0)$ nor
any $h_{2,a,b}=\LM_{S^G}(h_{2,a,b})$
for $1\leq a \leq b < n$
divide $\LM_{S^G}(f)$.
We show $\LM_{S^G}(h_{1,j})$
divides $\LM_{S^G}(f)$ for some $1 \leq j < n$.
We write
\begin{equation*}
    \LM_{S^G}(f)=f_1^{c_1}f_2^{c_2}\cdots f_n^{c_n}
\end{equation*}
for some $c_n< \deg_{f_n}(h_0)
= 1+(p^m-1)/e$
and
some $c_1, 
\ldots, c_{n-1}$.
But $f$ and hence 
$\LM_S(f)$ lies in $\Ipm$,
so
$p^{m-1} \leq c_j $
for some index $j<n$
since
$$
\LM_S(f)=\LM_{S}\big( \LM_{S^G}(f)\big)= x_1^{pc_1}x_2^{pc_2}\dots  x_{n-1}^{pc_{n-1}}
x_n^{ec_n}\, .
$$
Then $f_j^{p^{m-1}}$
divides $\LM_{S^G}(f)$.
~\cref{lemma: leading monomial division}
implies that $\LM_{S^G}(f)$ 
is also divisible by $f_n$, hence
by $\LM_{S^G}(h_{1,j})=f_{j}^{p^{m-1}}f_n$ as well.
As $\mathscr{G}\subset S^G \cap \Ipm$ by \cref{lemma:definition of invariants in frobenius ideal}, $\mathscr{G}$ is a Groebner basis for $S^G\cap \Ipm$.

\end{proof}

\begin{example}
{\em
For
$
    G=\left<\Big(\begin{smallmatrix} 1 & 0 & 0 \\
    0 & 1 & 0\\
    0 & 0 & \omega \end{smallmatrix}\Big), \Big(\begin{smallmatrix} 1 & 0 & 1 \\
    0 & 1 & 0 \\
    0 & 0 & 1
    \end{smallmatrix}\Big), \Big(\begin{smallmatrix} 1 & 0 & 0 \\
    0 & 1 & 1 \\
    0 & 0 & 1
    \end{smallmatrix}\Big) \right>
    \subset\GL_3(\FF_5)
    ,
$
the polynomials
\begin{equation*}
   \begin{aligned}
    &h_0 = f_3^{1+e^{-1}(5^m-1)}\, ,
    \quad
    \, h_{1,1} = \ \sum_{k=0}^{m-1}f_3^{1+e^{-1}(5^m-5^{m-k})}\ f_1^{5^{m-k-1}}\, ,
    \quad
    h_{2,2,2}=f_2^{2(5^{m-1})},
    \\
    &h_{1,2} = \ \sum_{k=0}^{m-1}f_3^{1+e^{-1}(5^m-5^{m-k})}\ f_2^{5^{m-k-1}}\, ,
    \quad
    h_{2,1,1} = f_1^{2(5^{m-1})}\, , 
    \ \ \text{ and } \ \  
    h_{2,1,2}=f_1^{5^{m-1}}f_2^{5^{m-1}}\, 
\end{aligned}
\end{equation*}
form a Groebner basis for $S^G\cap \mathfrak{m}^{[5^m]}$ as an ideal of $S^G$.
}
\end{example}

\section{Hilbert Series
of invariants in 
the Frobenius Irrelevant Ideal}
\label{section: hilbert series A}

Again, we assume $G$ is a subgroup of $\Glp$ fixing a hyperplane $H$
and set $e$ to be the 
maximal order of a semisimple
element of $G$.  
We consider the case when $G$
has maximal transvection root space,
i.e., the case when
$G$ is generated by $n-1$
transvections 
together possibly with a
semisimple reflection of order $e$. 
For any graded module $M$, 
we write $M[i]$ for the graded module
with degrees shifted down
by $i$ so that $M[i]_d = M_{i+d}$.

\begin{prop}
\label{lemma:hilbert series of invariant mod ideal}
Suppose $G$ has maximal transvection root space.
Then
\[
\begin{aligned}
\Hilb\Big(
\faktor{S^{G} +\Ipm}{\Ipm},\ t\Big) 
=&\
\Big(
\mfrac{1-t^{p^m}}
{1-t^p}
\Big)^{n-1}\ 
\Big(\mfrac{
1-t^{p^m+e-1}+(n-1)\, t^{p^m}(1-t^e)
}{1-t^e} \Big)
\, .
\end{aligned}
\] 
\end{prop}

\begin{proof}
We replace the ideal $S^G\cap \Ipm$
by its initial ideal
with respect to the graded lexicographical order on
$\FFp[x_1, \ldots, x_n]$
with $x_1>\cdots > x_n$,
since (see, for example,~\cite{Eisenbud})
\[
\Hilb\Big(
\faktor{S^{G} +\Ipm}{\Ipm},\ t\Big) 
=
\Hilb\Big(\faktor{S^{G}}{ (S^{G} \cap 
\Ipm)}\, ,\ t\Big) 
= 
\Hilb\Big(\faktor{S^{G}}{ \textrm{in}(S^{G} \cap \Ipm)}\, ,\ t\Big)\, .
\]
We compute the Hilbert series recursively using
short exact sequences.
By
\cref{groebner basis},
$\mathscr{G}$ 
is a Groebner basis
for $S^G\cap \Ipm$, 
and we enumerate 
the various elements $h_0$, $h_{1,a}$,
and $h_{2,a,b}$
in $\mathscr{G}$
as $h_1, h_2, h_3, \ldots,
h_{n-1+\binom{n}{2}}$
by setting
\begin{equation*}
  h_k = h_{1,k}
\quad\text{ for } 1\leq k <n
\quad\text{ and }\quad
h_{na+b-\binom{a+1}{2}}
=
h_{2,a,b} 
\quad\text{ for }
1\leq a \leq b < n\, .
\end{equation*}

Set $M_{n+\binom{n}{2}}=1$ (for ease with notation),
$I_0 = \big(M_0\big)$,
$M_i=\LM_{S^G}(h_i)
$ and 
\begin{equation*}
    J_i=\big(M_j/ \gcd(M_j, M_{i+1}): 0 \leq j \leq i\big)\quad
     \text{ for }
      0 \leq i \leq n-1+\tbinom{n}{2}. 
\end{equation*}
Note that $I_{n-1+\binom{n}{2}}=\text{in}(S^G\cap\Ipm).$
This gives the short exact sequence
(for each $i$)
\begin{equation}\label{shortexseq}
    0 \longrightarrow \big(\faktor{S^{G}}{J_i}\big)\big[-\deg(M_{i+1})\big]
\longrightarrow \faktor{S^{G}}{I_i} \longrightarrow \faktor{S^{G}}{I_{i+1}} \longrightarrow 0\, .
\end{equation}

Each ideal 
$I_i$ is
uniquely determined by
some polynomial $h_i$
of the form $h_0$, $h_{1,a}$, or 
$h_{2,a,b}$,
and we revert to more suggestive
notation for the next computations,
defining
$$
\begin{aligned}
&I^0=I_0, \ \ 
&I^{1,k} = I_{k},\ \ 
&I^{2,a,b} = I_{na+b-\binom{a+1}{2}}
\quad\text{ for } 1\leq k < n,\
1\leq a \leq b < n\, ;
\\
&J^0=J_0, \ \ 
&J^{1,k} = J_{k},\ \
&J^{2,a,b} = J_{na+b-\binom{a+1}{2}}
\quad\text{ for } 1\leq k < n,\
1\leq a \leq b < n\, ,
\end{aligned}
$$
so that the ideals
$I_1, I_2, \ldots, I_{n-1+\binom{n}{2}}$
merely enumerate the ideals
$I^0, I^{1,a}, I^{2,a,b}$
for ease with induction,
with the last ideal
in our sequence just
\begin{equation*}
    I^{2,n-1,n-1}=I_{n-1+\binom{n}{2}}=\textrm{in}(S^G \cap \Ipm)\, .
\end{equation*}
To find the Hilbert series for
$S^G/ J_i$,
we first 
give minimal generating sets for each 
$J_i$,
\begin{equation*}
    \begin{aligned}
        &J^0
        &=&\ \big(f_n^{e^{-1}(p^m-1)}\big),&&  && \\
        &J^{1,a}
        &=&\
        \big(f_n^{e^{-1}(p^m-1)}, f_j^{p^{m-1}}: 1 \leq j \leq a\big)
        &&\textrm{for }\ 1 \leq a \leq n-2,\\
        &J^{1,n-1} 
        &=&\ (f_n), && \\
        &J^{2,a,b} &=&\ \big(f_n, f_j^{p^{m-1}}: 1\leq j \leq b \big) &&\textrm{for }\ 1 \leq a \leq b \leq n-2,\\
        &J^{2,a,n-1} 
        &=&\ \big(f_n, f_j^{p^{m-1}}: 1\leq j \leq a \big) &&\textrm{for }\ 1 \leq a \leq n-2\, ,
    \end{aligned}
\end{equation*}
\hspace{-1ex}
and then use the additivity of Hilbert series
over
short exact sequences
of the form
\begin{small}
\begin{equation*}
\begin{aligned}
     0&\! \longrightarrow (f_n^d)\! \longrightarrow S^{G}\! \longrightarrow \faktor{S^{G}}{(f_n^d)}\! \longrightarrow 0
     \quad\quad\quad\text{ and }
\\
    0&\! \longrightarrow \faktor{S^G\!\!\!\!}{\!\!(f_n^d, f_i^{p^{m-1}}\!\!\!\!:1 \leq i < c)}[-p^{m}]\! \longrightarrow \faktor{S^G\!\!\!\!}{\!\!(f_n^d, f_i^{p^{m-1}}\!\!\!\!: 1 \leq i < c)}\! \longrightarrow \faktor{S^G\!\!\!\!}{\!\!(f_n^d, f_i^{p^{m-1}}\!\!\!\!:1 \leq i \leq c)}\! \longrightarrow 0\, 
    \end{aligned}
\end{equation*}
\end{small}
\hspace{-2ex}
(for $d =1$ or $d=e^{-1}(p^m-1)$ and $1\leq c \leq n-1$).
We conclude that
\begin{equation}
\label{equation: Hilbert series of J ideals}
    \begin{aligned}
        &\Hilb\Big(\faktor{S^G}{J^{0}}, \ t \Big)& 
        &=&\ 
        &\mfrac{1-t^{p^m-1}}
        { (1-t^e)(1-t^{p})^{n-1}}\, ,& &&  \\
        &\Hilb\Big(\faktor{S^G}{J^{1,a}}, \ t \Big)& 
        &=&\ 
        &\mfrac{(1-t^{p^m})^a(1-t^{p^m-1})}
        {(1-t^e)(1-t^{p})^{n-1}}& &\textrm{for }\ 1 \leq a \leq n-2\, ,&
        \\               &\Hilb\Big(\faktor{S^G}{J^{1,n-1}}, \ t \Big)&
        &=&\ &\mfrac{1-t^{e}}
        {        (1-t^e)(1-t^{p})^{n-1}}\, ,& && \\
        &\Hilb\Big(\faktor{S^G}{J^{2,a,b}}, \ t \Big)& 
        &=&\ &\mfrac{(1-t^{p^m})^b(1-t^{e})}
        {        (1-t^e)(1-t^{p})^{n-1}}& 
        &\textrm{for }\ 1 \leq a \leq b \leq n-2\quad\text{ and},& \\
        &\Hilb\Big(\faktor{S^G}{J^{2,a,n-1}}, \ t \Big)& 
        &=&\ &\mfrac{(1-t^{p^m})^a(1-t^{e})}
        {                (1-t^e)(1-t^{p})^{n-1}}& 
        &\textrm{for }\ 1 \leq a \leq n-2\, .&
    \end{aligned}
\end{equation}
Then as
\begin{equation}
\label{equation: Hilbert series of first I}
    \Hilb\Big(\faktor{S^G}{I^{0}}, \ t \Big)
    = \mfrac{1-t^{p^m+e-1}}{\rule[1.2ex]{0ex}{1ex}(1-t^e)(1-t^{p})^{n-1}}
    \, ,
\end{equation}
\Cref{shortexseq,equation: Hilbert series of J ideals,equation: Hilbert series of first I} imply that
$\Hilb\big(S^{G}/ \text{in}(S^G\cap \Ipm)\, ,\, t\big)$ is
\begin{equation*}
    \begin{aligned}
           & \underbrace{      \mfrac{1-t^{p^m+e-1}}{
        \rule[1ex]{0ex}{1ex}
        (1-t^e)(1-t^{p})^{n-1}}
        \rule[-3ex]{1ex}{0ex}}_{I^0}
           -\underbrace{t^{p^m+e}
            \mfrac{1-t^{p^m-1}}{
            \rule[1.2ex]{0ex}{1ex}
            (1-t^e)(1-t^{p})^{n-1}}
            \rule[-3ex]{1ex}{0ex}}_{J^0}
        -\underbrace{t^{p^m+e}
        \sum_{a=1}^{n-2}
        \mfrac{(1-t^{p^m})^a(1-t^{p^m-1})}{
        \rule[1ex]{0ex}{1ex}(1-t^e)(1-t^{p})^{n-1}}}_{J^{1,a} 
        }
                        \\
        \rule[-8ex]{0ex}{13ex}
        &
        -\underbrace{t^{2p^m}
        \mfrac{1-t^{e}}{
        \rule[1ex]{0ex}{1ex}(1-t^e)(1-t^{p})^{n-1}}
        \rule[-3.1ex]{1.5ex}{0ex}}_{J^{1,n-1}} - \underbrace{t^{2p^m}
        \sum_{b=1}^{n-2}\sum_{a=1}^{b}
        \mfrac{(1-t^{p^m})^b(1-t^{e})}{
        \rule[1ex]{0ex}{1ex}(1-t^e)(1-t^{p})^{n-1}}}_{J^{2,a,b} 
        }
        %
        -\underbrace{t^{2p^m}
        \sum_{a=1}^{n-2}
        \mfrac{(1-t^{p^m})^a(1-t^{e})}{
        \rule[1ex]{0ex}{1ex}(1-t^e)(1-t^{p})^{n-1}}}_{J^{2,a,n-1} 
        }
        \ .
    \end{aligned}
\end{equation*}
We combine summations to express
$\Hilb\big(
S^{G}/ \text{in}(S^G\cap \Ipm)\, ,\, t\big)$ as
\begin{equation*}
    \begin{aligned}
&\mfrac{1-t^{p^m+e-1}}{
\rule[.5ex]{0ex}{1ex}
(1-t^e)(1-t^{p})^{n-1}}
        -t^{p^m+e}
        \sum_{a=0}^{n-2}
        \mfrac{(1-t^{p^m})^a(1-t^{p^m-1})}{\rule[.7ex]{0ex}{1ex}(1-t^e)(1-t^{p})^{n-1}}\\
        &\hspace{6ex}- t^{2p^m}(1-t^{e})
        \sum_{b=1}^{n-2}\sum_{a=1}^{b}
        \mfrac{(1-t^{p^m})^b}
        {\rule[.5ex]{0ex}{1ex}
        (1-t^e)(1-t^{p})^{n-1}}
        -t^{2p^m}
        \sum_{a=0}^{n-2}
        \mfrac{(1-t^{p^m})^a(1-t^{e})}{
        \rule[.5ex]{0ex}{1ex}
        (1-t^e)(1-t^{p})^{n-1}}\, ,  
    \end{aligned}
\end{equation*}
which simplifies (using elementary series formulas)
to
\begin{equation*}
    \begin{aligned}
       & \mfrac{1-t^{p^m+e-1}}{
       \rule[.5ex]{0ex}{1ex}
       (1-t^e)(1-t^{p})^{n-1}}
        -t^{p^m+e}\mfrac{\big(1-(1-t^{p^m})^{n-1}\big)(1-t^{p^m-1})}{
        \rule[.5ex]{0ex}{1ex}
        t^{p^m}(1-t^e)(1-t^{p})^{n-1}}\\
        &\hspace{9ex}- t^{2p^m}(1-t^{e})\sum_{b=1}^{n-2}\mfrac{b(1-t^{p^m})^b}{
        \rule[.5ex]{0ex}{1ex}
        (1-t^e)(1-t^{p})^{n-1}}
        -t^{2p^m}\mfrac{\big(1-(1-t^{p^m})^{n-1}\big)(1-t^{e})}{
        \rule[.5ex]{0ex}{1ex}
        t^{p^m}(1-t^e)(1-t^{p})^{n-1}}\ .
    \end{aligned}
\end{equation*}
We use the
fact that
\begin{equation*}
    \sum_{b=1}^{n-2}
    b\, (1-t^{p^m})^b 
    \ =\ 
    \mfrac{-(1-t^{p^m})
    (
    (n-1)(1-t^{p^m})^{n-2}\
    t^{p^m}+1-(1-t^{p^m})^{n-1})
    }{t^{2p^m}
    \rule[.2ex]{0ex}{1.7ex}} \  
\end{equation*}
to rewrite this last expression as
\begin{equation*}
    \Hilb\Big(\faktor{S^{G}}{\text{in}(S^G\cap \Ipm)}\, ,\ t\Big) 
    \ =\ 
    \mfrac{(1-t^{p^m})^{n-1}\big(1-t^{p^m+e-1}+(n-1)t^{p^m}(1-t^e)\big)}{(1-t^e)(1-t^p)^{n-1}}\, .
\end{equation*}
\end{proof}


\begin{example}
{\em For 
$
    G=\left<\Big(\begin{smallmatrix} 1 & 0 & 0 \\
    0 & 1 & 0\\
    0 & 0 & \omega \end{smallmatrix}\Big), \Big(\begin{smallmatrix} 1 & 0 & 1 \\
    0 & 1 & 0 \\
    0 & 0 & 1
    \end{smallmatrix}\Big), \Big(\begin{smallmatrix} 1 & 0 & 0 \\
    0 & 1 & 1 \\
    0 & 0 & 1
    \end{smallmatrix}\Big) \right>
\subset \GL_3(\FF_5)$,
\cref{lemma:hilbert series of invariant mod ideal} gives
\begin{equation*}
   \begin{aligned}
     \Hilb\Big(
     \faktor{(S^G+\mathfrak{m}^{[5^m]})}
     {\mathfrak{m}^{[5^m]}
     }
     , \, t\Big)
     \ =\
     \mfrac{
(1-t^{5^m})^{2}
\big(1-t^{5^m+e-1}+ 2t^{5^m}(1-t^e)\big)}
{(1-t^e)(1-t^5)^{2}} 
\, . 
\end{aligned}
\end{equation*}}
\end{example}

\section{Decomposition of the invariant space}
\label{section: direct sum decomp}
We now use the description of $S^{G} \cap \Ipm$
from the last two sections to give a direct sum decomposition of 
$(S/ \Ipm)^{G}$.
Again, we consider a subgroup $G$ 
of $\Glp$ fixing a hyperplane
and 
use the basis $x_1, \ldots,
x_n$ of $V^*$ 
and basic invariants
$f_1, \ldots, f_n$ as in \cref{choosebasis}.
We show that 
$(S/ \Ipm )^{G}$
is the direct sum of subspaces
$$
\begin{aligned}
A_G=& \ \faktor{(S^{G}+\Ipm)}{\Ipm}
\quad\text{ and }\\
B_G=&\ 
\mathbb{F}_p[f_1, \dots , f_{n-1}]
\textrm{-span}
\Big\{x_1^{a_1}\dots x_{\ell}^{a_{\ell}}
x_n^{p^m-1}+\Ipm: 0 \leq a_i < p,\sum_{i=1}^{\ell}a_i \geq 2\Big\}.
\end{aligned}
$$ 
Note that
$B_G$ is the $\mathbb{F}_p[f_1, \dots , f_{n-1}]$-submodule
of $S/\Ipm$
spanned by the monomial
cosets indicated.
Recall that
$\ell$ is the minimal number of transvections
generating $G$ together
with
a semisimple element
of order $e$; if no group elements are semisimple, $e=1$.

\begin{remark}{\em
In  defining the subspace $B_G$, we require  $\sum_{i=1}^{\ell}a_i \geq 2$ to avoid nontrivial intersection with $A_G$; see
\cref{corollary:description of invariant space}. Otherwise $B_G$ would contain $x_n^{p^m-1}+\Ipm$ and $x_ix_n^{p^m-1}+\Ipm$, for example,
which lie in $A_G$ for
$i<\ell$.

}
\end{remark}

We first describe the leading monomial in $(S/ \Ipm)^{G}$ using the standard graded lexicographical order on $S=\mathbb{F}_p[x_1, \dots, x_n]$
with $x_1>\cdots> x_n$.
\begin{lemma}
\label{lemma: description of initial term of invariants mod ideal in general case}
Assume $G$ has maximal transvection
root space.
Suppose $f+\Ipm$ lies in $(S/\Ipm)^G$
with $f$ homogeneous in
$x_1, \ldots, x_n$.
Then $\LM_S(f)$
lies in
$$
\Ipm
\qquad{ or }\quad
\FF_p[x_1, \ldots, x_{n-1}, f_n^{e^{-1}(p^{m}-1)}]
\qquad{ or }\quad
\FF_p[x_1^p, \ldots, x_{n-1}^p, 
f_n]\, .
$$
\end{lemma}

\begin{proof}
Say $M=\LM_S(f)$ does not lie in $\Ipm$ or in $\FFp[x_1, \ldots, x_{n-1}, f_n^{e^{-1}(p^{m}-1)}].$
Then
\begin{equation*}
    M=
   x_1^{b_1}\cdots  x_k^{b_k}
   \cdots x_{n-1}^{b_{n-1}}\,
   x_n^{b_n}
   \quad
   \text{ for some }
      \ 
   b_1, \ldots, b_{n-1}
   <p^m 
   \text{ and }
   b_n< p^m-1\, .
\end{equation*}
We use the generators $g_1, \ldots, g_n$
of $G$ from \cref{choosebasis}.
Since $f$ is $G$-invariant
modulo $\Ipm$, the difference
$g_n f-f$ lies in $\Ipm$
and the low degree of each
$x_i$ forces $f$ itself to be
invariant under  $g_n$;
hence 
$b_n$ is divisible by $e$.

Suppose there is some exponent $b_k$ 
which is not divisible by $p$ 
with $k<n$.
Consider $g=g_k^{-1}$ acting on 
$M=\LM_S(f)$. Then 
$g\cdot M-M$ is
\begin{equation*}
   \begin{aligned}
       x_1^{b_1}\cdots x_{k-1}^{b_{k-1}}
       \big(
       (x_k+x_n)^{b_k}-
       x_k^{b_k}
       \big)\,
       x_{k+1}^{b_{k+1}}
       \cdots x_n^{b_n}
    \end{aligned}
\end{equation*}
with leading monomial
\begin{equation}
    \label{mono}
\LM_S(g M - M)
=
x_1^{b_1}\cdots x_{k-1}^{b_{k-1}}\,
(b_k\, x_k^{b_k-1})\, 
x_{k+1}^{b_{k+1}}
x_{k+2}^{b_{k+2}}\cdots x_{n-1}^{b_{n-1}}\, x_n^{b_n+1}\, 
\end{equation}
as $b_k\neq 0$ in $\FFp$.
Notice that
the leading monomial of 
$gM-M$ is the leading
monomial of $gf -f$
as $M$ is the leading monomial of
$f$ and  $g$ fixes $x_1, \ldots,
x_{k-1}, x_{k+1}, \ldots, x_n$.

Since $f$ is invariant modulo $\Ipm$,
the difference $gf-f$ and thus
its leading monomial 
(\ref{mono}) lie in $\Ipm$.
But this is impossible as
$b_n< p^m-1$ 
and 
$b_1, \ldots, b_{n-1}
<p^m$ by our assumptions.
Thus $p$ must divide every exponent $b_k$ for $k<n$.
\end{proof}

We use ~\cref{lemma: description of initial term of invariants mod ideal in general case} to decompose
the invariants of the quotient space.
\begin{prop}
\label{proposition:A+B}
For $G$ with maximal transvection root space,
$\big(\faktor{S}{\Ipm}\big)^{G}=A_G + B_G$. 
\end{prop}

\begin{proof}
By construction, $A_G \subseteq (S/ \Ipm)^{G}$.  To show $B_G \subseteq (S/ \Ipm)^{G}$, consider
$$
M = x_1^{a_1}\dots x_{n-1}^{a_{n-1}}x_n^{p^m-1}
\quad\text{ with }
M + \Ipm \in B_G.
$$ 
We consider the generators 
$g_1, \ldots, g_n$
of $G$ from \cref{choosebasis};
for $k<n$,
\begin{equation*}
    \begin{aligned}
        g_k^{-1} ( M) + \Ipm 
        & = 
        x_1^{a_1} \dots x_{k-1}^{a_{k-1}}(x_k+x_n)^{a_k}x_{k+1}^{a_{k+1}} \dots x_{n-1}^{a_{n-1}}x_n^{p^m-1} + \Ipm\\
        &= x_1^{a_1}\dots x_{n-1}^{a_{n-1}}x_n^{p^m-1} + \Ipm
        \ =\ M + \Ipm,
    \end{aligned}
\end{equation*}
since the binomial theorem implies that all but the initial term lies in $\Ipm$.
In addition, $e$ divides $p^m-1$, 
so $g_n$ fixes $M$. Hence $B_G$ is $G$-invariant and thus $A_G+B_G \subseteq (S/\Ipm)^{G}$.

To show the reverse containment, we first argue that any monomial $M$ in the variables
$x_1, \ldots, x_{n}$
with $\deg_{x_n}(M)={p^m-1}$
represents a coset of $\Ipm$
either in 
$A_G$ or in $B_G$.  
Both $A_G$ and $B_G$ are closed
under multiplication by
$f_1, \ldots, f_{n-1}$, 
so
we assume 
without loss of generality that
$\deg_{x_i}(M)<p$ for
$i<n$  by \cref{modtheideal}.
Let 
$
k=\sum_{i=1}^{n-1}
\deg_{x_i}(M) \, .
$
If $k\geq 2$,
then $M+\Ipm$ lies in $B_G$
by definition.
If $k=0$,
then $M=x_n^{p^m-1}
=
    f_n^{e^{-1}(p^m-1)} 
    $
and $M+\Ipm$ lies in $A_G$.
 If $k=1$, then $M+\Ipm$ lies
 in $A_G$ as well since
 \begin{equation*}
    -x_i\, x_n^{p^m-1}
    \equiv 
   \sum_{j=0}^{m-1}f_n^{e^{-1}(p^m-p^j)}f_i^{p^{m-1-j}}
        \mod\Ipm
    \quad\text{for } i < n.
\end{equation*}

If the reverse containment fails,
we may 
choose some $f+\Ipm$ in 
$(S/\Ipm)^G$
but not in $A_G+B_G$
with $f$
homogeneous in
$x_1, \ldots, x_{n}$ 
and $\LM_{S}(f)$ minimal. 
Note that
 $\deg_{x_i}(\LM_{S}(f))<p^m$
 for all $i$.
By the minimality assumption,
$\LM_{S}(f) + \Ipm$ does
not lie in $A_G$ or $B_G$,
so by the argument in the last
paragraph,
 $\deg_{x_n}(\LM_{S}(f)) < p^m-1$. 
By ~\cref{lemma: description of initial term of invariants mod ideal in general case}, the monomial
$\LM_S(f)$ lies in 
$\FFp[x_1^p, \ldots,
x_{n-1}^p, f_n]$,
so
\begin{equation*}
    \LM_S(f)=x_1^{pc_1}x_2^{pc_2}\cdots x_{n-1}^{pc_2}
    \,x_n^{ec_n}
\quad\text{ for some } c_i \, .
\end{equation*}
Define $h$ by
\begin{equation*}
    h=\alpha f_1^{c_1}f_2^{c_2}\cdots f_{n-1}^{c_{n-1}}
    \,f_n^{c_n}, \quad \text{for } \alpha \text{ the leading coefficient of } f.
\end{equation*}
Then  $f-h+\Ipm$ lies
in $(S/\Ipm)^G$ since $h+\Ipm$ lies in $A_G$, and,
by construction,
$
    \LM_S(h)=\LM_S(f)\, ,
$
implying that $\LM_S(f-h)<\LM_S(f)$. 
The minimality assumption then implies that
$f-h$ must lie in $A_G+B_G$. However, $A_G+B_G$ contains $h$ already,
so must contain $f$ as well,
contradicting our choice of $f$.
Thus $(S/\Ipm)^G=A_G+B_G$.
\end{proof}


\begin{prop}
\label{corollary:description of invariant space}
Suppose 
the transvection root space
of $G$ is maximal.
Then
$$\big(\faktor{S}{\Ipm}\big)^{G} = A_G \oplus B_G\, .
$$ 
\end{prop}
\begin{proof}
By~\cref{proposition:A+B},
we need only show $A_G\cap B_G$
is trivial.
If $m \leq 1$,
a simple degree comparison shows $A_G\cap B_G=\{0\}$,
hence we assume $m \geq 2$.
Suppose $A_G \cap B_G$ is 
non-trivial, say some
$f$ in $S^G$
and $h + \Ipm$ in $B_G$ satisfy
\begin{equation*}
    0\neq f+\Ipm = h+ \Ipm \in A_G \cap B_G\, .
\end{equation*}
We multiply $f-h$ by $f_n=x_n^e$ so that 
$
    (f-h)f_n $
    lies in $\Ipm \, f_n$.
We will show that $f f_n$ and $h f_n$
have no monomials in the variables
$x_1, \ldots, x_n$ in common;
this will force $hf_n$ to
lie in 
 $\Ipm f_n$, 
 contradicting the fact
 that $h$ does not lie in $\Ipm$.
 
Fix some $M$ in $X_f \cap X_h$
for
$$
\begin{aligned}
         &X_f  
         \text{ the set of
         monomials in $x_1, \ldots, x_n$ of } f f_n,
         \text{ and }\\
         &X_h  
         \text{ the set of
         monomials in $x_1, \ldots, x_n$ of } h f_n
         \, .
         \end{aligned}
$$
Since $h$ lies in the ideal $(x_n^{p^m-1})$
and $e\geq 1$, the ideal
$\Ipm$ contains
$h f_n$ and 
thus also $ f f_n = (f-h)  f_n + h  f_n$.
However, $f f_n$ also lies in $ S^G$, so 
$ f  f_n$ lies
    in    $S^G \cap \Ipm$.
\cref{groebner basis} then implies 
that
$M$ is a monomial
of some  $S^G$-multiple
of $h_0$, $h_{1,a}$,
or $h_{2,a,b}$
for some
$1 \leq a \leq b < n$
(see~\cref{definition:invariants}) 
and we use ~\cref{lemma:definition of invariants in frobenius ideal} to expand in the variables 
$x_1, \ldots, x_n$.
Since $h$ is not in $\Ipm$ and $M$ lies in $X_h$, \cref{modtheideal} implies that
    \begin{align}
    \deg_{x_n}(M) &= p^m+e-1
    \quad \textrm{ and }\label{M eqn a}
    \\
    \ \deg_{x_i}(M) &= b_i\, p+a_i 
    \text{ for some } b_i < p^{m-1}, a_i < p, \text{ with } \textstyle{\sum_{i=1}^{n-1}}
    a_i \geq 2\, . \label{M  eqn b}
    \end{align}

First, say $M$ is a monomial of some polynomial in $S^G\, h_0$. 
By  ~\cref{lemma:definition of invariants in frobenius ideal},
for $i<n$,
\begin{equation*}
    \begin{aligned}
    \deg_{x_n}(M)&=x_n^{p^m+e-1+ c_n +(p-1)\sum_{i=1}^{n-1}j_i}& 
     \ &\text{ and}&\\
    \deg_{x_i}(M) &= pc_i-(p-1)j_i 
    = ( c_i - j_i)p + j_i&
     &\textrm{ for some  }   c_i \in \NN  \textrm{ and } 0 \leq j_i \leq c_i\, .
        \end{aligned}
\end{equation*}
But ~\cref{M eqn a} implies
that $j_i=0$ for all $i<n$ and $c_n=0$. Then $p$ must  divide $\deg_{x_i}(M)$ for each $1 \leq i \leq n$, contradicting ~\cref{M eqn b}.

Second, say that $M$
is a monomial of some polynomial in
$S^G\, h_{1,a}$ for
some $a<n$. Without loss of generality,  say $a=1$. Then, for $1 < i < n$,
\begin{equation*}
    \deg_{x_i}(M) = pc_i-(p-1)j_i \ \quad\textrm{ for some  }  \ c_i \in \NN \ \textrm{ and } 0 \leq j_i \leq c_i\, .
\end{equation*}
 Furthermore, by ~\cref{lemma:definition of invariants in frobenius ideal},
\begin{equation*}
    \deg_{x_1}(M) = p^m+pc_1-(p-1)j_1 \quad\quad \textrm{ or } \quad\quad \deg_{x_1}(M) = 1+pc_1-(p-1)j_1
\end{equation*}
for some $c_1 \in \NN$ and $1 \leq j_1 \leq c_1$. 
But ~\cref{M eqn b}
implies the latter case holds,
and thus
\begin{equation*}
    \deg_{x_n}(M)=p^m+e-1 +c_n+(p-1)
    \textstyle{\sum_{i=1}^{n-1}}
    j_i \quad \text{ for some } c_n \in \NN \, .
\end{equation*}
Again, ~\cref{M eqn a} implies 
$j_i = 0$ for all $1\leq i< n $ and $c_n=0$. However, this forces $p$ to divide $\deg_{x_i}(M)$ for $2 \leq i < n$ and $\deg_{x_1}(M) $ to be $1+ pc_1$, contradicting ~\cref{M eqn b}. 

Third, say that $M$ is a monomial of some polynomial in $S^G\, h_{2,a,b}$
for some pair $a,b$
with $1\leq a \leq b < n$. 
~\cref{M eqn b} implies 
that the degree of ${x_a}$ or
of ${x_b}$
in each monomial
of 
$h_{2,a,b}$ is too high
except 
the last monomial
$N=x_a^{p^{m-1}}
x_b^{p^{m-1}}
x_n^{2(p-1)p^{m-1}}$.
But for any monomial $M'$ appearing
in an
$S^G$-multiple of $N$,
$p$ divides $\deg_{x_i}(M')$
for all $i<n$ or $\deg_{x_n}(M')>p^m+e-1$,
contradicting
~\cref{M eqn b} and ~\cref{M eqn a}. 
(One can check the case $p=2$ separately.)
\end{proof}

\begin{example}
{\em
For
$
    G=\left<\Big(\begin{smallmatrix} 1 & 0 & 0 \\
    0 & 1 & 0\\
    0 & 0 & \omega \end{smallmatrix}\Big), \Big(\begin{smallmatrix} 1 & 0 & 1 \\
    0 & 1 & 0 \\
    0 & 0 & 1
    \end{smallmatrix}\Big), \Big(\begin{smallmatrix} 1 & 0 & 0 \\
    0 & 1 & 1 \\
    0 & 0 & 1
    \end{smallmatrix}\Big) \right>
    \subset \GL_3(\FF_5)$,
~\cref{corollary:description of invariant space}
implies that the space
$(S/\mathfrak{m}^{[5^m]})^G$
decomposes 
as 
\begin{equation*}
   \begin{aligned}
\faktor{(S^G + \mathfrak{m}^{[5^m]})}
{\mathfrak{m}^{[5^m]}}\ \oplus\ \mathbb{F}_5[f_1,f_2]
\textrm{-span}\{x_1^{a_1}x_2^{a_2}x_3^{5^m-1}
+\mathfrak{m}^{[5^m]}: a_i < 5, \, a_1+a_2 \geq 2\}\, .
\end{aligned}
\end{equation*}
}
\end{example}

In the next result,
we do not assume
the transvection root
space is maximal.
\begin{cor}
\label{corollary:description of invariant space general case}
For any group $G\subset\Glp$ fixing a hyperplane,
$$\big(\faktor{S}{\Ipm}\big)^{G} = A_G \oplus B_G\, .
$$
\end{cor}
\begin{proof}
We  decompose the vector space $V$ 
to separate out the trivial action:
set
$$
V_1
=\CC\text{-span}\{v_1, \ldots,v_\ell, v_n\} 
\quad\text{and}\quad V_2=\CC\text{-span}\{v_{\ell+1}, 
\ldots,v_{n-1}\}\, ,
$$
and set $S_1=S(V_1^*)=\FFp[x_1, \ldots, x_\ell, x_{n}]$
and $S_2=S(V_2^*)=\FFp[x_{\ell+1}, \ldots, x_{n-1}]$.
Likewise, set
$\Ipm_1=
(x_1^{p^m}, \ldots,
x_\ell^{p^m},\, x_n^{p^m})$
and
$
\Ipm_2=(x_{\ell+1}^{p^m},\ldots,
x_{n-1}^{p^m})$ .
Then $G$ is the direct sum
$G=G_1 \oplus G_2$ for $G_i=G|_{V_i}$
and
$\Ipm=(\Ipm_1, \Ipm_2)$.
By \cref{corollary:description of invariant space},
$(S_1/\Ipm_1)^{G_1}
=A_{G_1}\oplus B_{G_1}$.
Since $G_2$
acts trivially
on $V_2$,
we may set
$A_{G_2}=(\FFp[v_{l+1},\ldots, v_{n-1}]+\Ipm_2)/\Ipm_2$
and 
$B_{G_2}=\{0\}$.
The graded isomorphism $S\cong S_1 \otimes_{\FFp} S_2$
induces a graded isomorphism
$$
\faktor{S}{\Ipm} 
\ \ \ \cong \ \ \
\faktor{S_1}{\Ipm_1}
\  \otimes_{\FFp} \   
\faktor{S_2}{\Ipm_2}
$$
and induces graded
vector space
isomorphisms
$$
\begin{aligned}
\Big(\faktor{S}{\Ipm}\Big)^G 
  \cong\,
\Big(\faktor{S_1}{\Ipm_1}\Big)^{G_1} 
\otimes_{\FFp}
\Big(\faktor{S_2}{ \Ipm_2}\Big)^{G_2} 
\, \cong\, 
(A_{G_1}\oplus B_{G_1})\ot_{\FFp}
A_{G_2}
 \cong\,  
A_G\oplus B_G\, .
\end{aligned}
$$
The result follows
since $A_G + B_G 
\subset
(S/\Ipm)^G$
(see the proof of
\cref{proposition:A+B}).
\end{proof}
\section{Hilbert Series 
for maximal transvection root spaces}\label{section: hilbert series maximal transvection}

Again, we assume throughout this section that
$G$ is a subgroup of $\Glp$ fixing a hyperplane $H$
and $e$ is the maximal order of a
semisimple element of $G$.
We assume the root space of $G$
is maximal to avoid excessive
notation arising from a trivial action of $G$ on extra variables.
By ~\cref{corollary:description of invariant space}, $(S \slash \mathfrak{m}^{[p^m]})^{G}$ is a direct sum $A_G \oplus B_G$ 
with invariant subspace $A_G$
described in
Sections~\ref{section: description of A}
and~\ref{section: hilbert series A}. 
For ease with
notation, we 
fix a basis of $V$ as in \cref{choosebasis}
and
describe here 
$$
B_G=\mathbb{F}_p[f_1, \dots , f_{n-1}]
\textrm{-span}
\Big\{x_1^{a_1}\dots x_{n-1}^{a_{n-1}}x_n^{p^m-1}+\Ipm:  0\leq a_i < p, \, \sum_{i=1}^{n-1}a_i \geq 2
\Big\}\, .
$$ 

\begin{lemma}
\label{lemma:hilbert series of B}
Suppose the transvection root space of $G$
is maximal.
Then
\[\Hilb(B_G,\ t) 
\ =\ 
t^{p^m-1}
\Big(\Big( \mfrac{1-t^p}{1-t\ } \Big)^{n-1} - (n-1) t - 1\Big)
\Big(\mfrac{1-t^{p^m}}{1-t^p\ } \Big)^{n-1} \, .
\]
\end{lemma}

\begin{proof}
Observe that
$
    B_G = 
    \FFp[f_1, \dots, f_{n-1}]
    \textrm{-span}\ 
    C
    \cong
     \FFp[f_1, \dots, f_{n-1}]
     \ot_{\FF_p} C
    $
    as a graded vector space
by \cref{modtheideal},
where
$$
C=
\mathbb{F}_p\textrm{-span}\{x_1^{a_1}\dots x_{n-1}^{a_{n-1}}x_n^{p^m-1} + \Ipm: 
0 \leq a_i < p, \,
a_1+\ldots+ a_{n-1} \geq 2\}
\, .
$$
Since
$\deg f_i=p$ for $i < n$,
$$
\begin{aligned}
\Hilb(B_G, \, t)
\ =&\ \ 
\Big(
\mfrac{\ \, 1- t^{p^m}}{1-t^p} \Big)^{n-1}
\ \cdot 
\ \
\Hilb(C, \, t)
\\
\ =&\ \ 
\Big(\mfrac{\ \, 1- t^{p^m}}{1-t^p} \Big)^{n-1}
\ t^{p^m-1}\Big(\Big(
\mfrac{\ \, 1-t^p}{1-t}\Big)^{n-1}
- (n-1) t - 1\Big)
\, ,
\end{aligned}
$$
with subtracted terms
arising from the restriction 
$a_1+ \ldots + a_{n-1} \geq 2$.
\end{proof}

\begin{theorem}
\label{theorem: hilbert series of total quotient space}
Suppose the transvection root space of $G$
is maximal.
Then
\begin{equation*}
    \begin{aligned}
        \Hilb\Big(\big(\faktor{S}{\Ipm}\big)^G,\ t \Big) 
        =
        \Big(
        \mfrac{1-t^{p^m}}{1-t^p\ }
        \Big)^{n-1}
        \Big(\mfrac{1-t^{p^m-1}}{
        1-t^e\ \ \ }\Big)
        +t^{p^m-1}
        \Big(\mfrac{1-t^{p^m}}{
        1-t\ \ }\Big)^{n-1} \ .
        \rule[-1.5ex]{0ex}{3ex}
    \end{aligned}
\end{equation*}
\end{theorem}
\begin{proof}
By~\cref{corollary:description of invariant space},
$(S \slash \Ipm )^{G} = A_G \oplus B_G$,
and the theorem follows 
from adding
the Hilbert series for $A_G$ and $B_G$ given in ~\cref{lemma:hilbert series of invariant mod ideal}  and ~\cref{lemma:hilbert series of B}
and simplifying.
\end{proof}

\begin{example}
{\em
For our archetype example,
$
    G=\left<\Big(\begin{smallmatrix} 1 & 0 & 0 \\
    0 & 1 & 0\\
    0 & 0 & \omega \end{smallmatrix}\Big), \Big(\begin{smallmatrix} 1 & 0 & 1 \\
    0 & 1 & 0 \\
    0 & 0 & 1
    \end{smallmatrix}\Big), \Big(\begin{smallmatrix} 1 & 0 & 0 \\
    0 & 1 & 1 \\
    0 & 0 & 1
    \end{smallmatrix}\Big) \right>
    \subset \GL_3(\FF_5)$,
\cref{lemma:hilbert series of B} implies 
 that
\begin{equation*}
   \begin{aligned}
     \Hilb( B_G, \, t)
     = t^{5^m-1}
\Big(\Big( \mfrac{1-t^5}{1-t\ } \Big)^{2} - 2t - 1\Big)
\Big(\mfrac{1-t^{5^m}}{1-t^5\
\rule[0ex]{0ex}{1.5ex}
} \Big)^{2}\, .
\end{aligned}
\end{equation*}
By~\cref{theorem: hilbert series of total quotient space} (recall $e=$ order($\om$)), the Hilbert series of 
$(S/\mathfrak{m}^{[5^m]})^G$  is
\begin{equation*}
   \begin{aligned}
     \Big(
        \mfrac{1-t^{5^m}}{1-t^5\ }
        \Big)^{2}
        \Big(\mfrac{1-t^{5^m-1}}{
        1-t^e\ \ \ }\Big)
        +t^{5^m-1}
        \Big(\mfrac{1-t^{5^m}}{
        1-t\ \ }\Big)^{2} \ .
\end{aligned}
\end{equation*}
}
\end{example}

We record
an alternate expression for the Hilbert series in
\cref{theorem: hilbert series of total quotient space}:
\begin{cor}
\label{MainThmMaxRootSpace}
Suppose the transvection root space of $G$
is maximal.
Then
\begin{equation*}
    \begin{aligned}
        \Hilb\Big(\big(\faktor{S}{\Ipm}\big)^G,\ t \Big)
         =& 
        \thilb(S^G, \, t)(1-t^{p^m})^{n-1}
        \Big( 1 - t^{p^m-1}\!
        +(1-t^e)t^{p^m-1}
        \Big(
        \mfrac{1-t^p}{1-t\ }
        \Big)^{\! n-1}\Big)
\\
    \ =& \ 
    \Big(\mfrac{ 1-t^{p^m}}{1-t^p\ \ }\Big)^{n-1}
    \Big(\Big(
    \mfrac{ 1-t^{p^m-1}}{1-t^e\ \ \ \ }\Big)
    +t^{p^m-1}
    \Big(
    \mfrac{1-t^p}{1-t\ } \Big)^{n-1}\Big) \ .
    \end{aligned}
\end{equation*}
\end{cor}

\section{Hilbert Series for arbitrary
group fixing a hyperplane}
\label{section:final result}

Again, we assume $G$ is a 
subgroup of $\Glp$ fixing a hyperplane $H$
and set
$
\ell= \dim_{\FFp}(
\text{RootSpace}(G))
\cap H,
$
the dimension of the
(not necessarily maximal) transvection root
space of $G$, with $e$ the maximal
order of a semisimple element of $G$.

\begin{theorem}
\label{mainthm}
Suppose $G$ is a 
subgroup of $\Glp$ fixing a hyperplane.
Then
\begin{equation*}
\begin{aligned}
 \Hilb \Big(  \big(\faktor{S}{\Ipm}\big)^G,\ t \Big)
&=\ 
\Big(\mfrac{1-t^{p^m}}
{1-t\ \ \ }\Big)^{n-\ell-1}\ 
\Big(
\mfrac{1-t^{p^m}}{1-t^p\ }\Big)^{\ell}\ 
\Big(\Big(
\mfrac{
1- t^{p^m-1}
}{1-t^e\ \ \ }\Big)
+
t^{p^m-1}
\Big( \mfrac{1-t^p}{1-t\ } \Big)^{\ell}
\Big)
\\
&=\
\Big(\mfrac{1-t^{p^m}}{1-t}
\Big)^{n-\ell-1}
\Big(\mfrac{1-t^{p^m}}{1-t^p}
\Big)^{\ell}
\Big(\mfrac{1-t^{p^m-1}}{1-t^e}
\Big)
+
t^{p^m-1}
\Big(\mfrac{1-t^{p^m}}{1-t}
\Big)^{n-1}
\\
&=\
\Hilb(S^G, t)\,
(1-t^{p^m})^{n-1}
\Big(
(1-t^{p^m-1}) +
t^{p^m-1}(1-t^{e})
\Big(\mfrac{1-t^p}{1-t}
\Big)^{\ell}
\Big)
\, .
\end{aligned}
\end{equation*}
\end{theorem}
\begin{proof}
We write $G=G_1\oplus G_2$,
$S=S_1\ot_{\FFp} S_2$,
and $\Ipm=(\Ipm_1, \Ipm_2)$
as in the proof of
\cref{corollary:description of invariant space general case}
and use 
the graded isomorphism
$$
\Big(\faktor{S}{\Ipm}\Big)^G 
\ \ \ \cong\ \ \ 
\Big(\faktor{S_1}{\Ipm_1}\Big)^{G_1} 
\otimes_{\FFp}\ 
\Big(\faktor{S_2}{ \Ipm_2}\Big)^{G_2}\, 
.
$$
Since $G_2$
acts trivially
on $V_2$ of dimension $n-\ell-1$, 
$$
\Hilb\Big(\big(
\faktor{S}{\Ipm_2}\big)^{G_2},
\ t\Big)
\ =\
\Big(\mfrac{1-t^{p^m}}
{1-t\ \ \ }\Big)^{n-\ell-1}\, .
$$
Since $G_1$ has maximal transvection root space in $V_1$,
~\cref{MainThmMaxRootSpace}
implies that
$$
\Hilb\Big(\big(
\faktor{S}{\Ipm_2}\big)^{G_1},
\ t\Big)
\ = \
\Big(\mfrac{1-t^{p^m}}
{1-t\ \ \ }\Big)^{n-\ell-1}\ 
\Big(
\mfrac{1-t^{p^m}}{1-t^p\ }\Big)^{\ell}\ 
\Big(\Big(
\mfrac{
1- t^{p^m-1}
}{1-t^e\ \ \ }\Big)
+
t^{p^m-1}
\Big( \mfrac{1-t^p}{1-t\ } \Big)^{\ell}
\Big)\, .
$$
The theorem then follows
from taking the product
of the two Hilbert series above.
\end{proof}

\begin{example}{\em
Say
$ G=\left<\Big(\begin{smallmatrix} 1 & 0 & 0 \\
    0 & 1 & 0\\
    0 & 0 & \omega \end{smallmatrix}\Big), \Big(\begin{smallmatrix} 1 & 0 & 1 \\
    0 & 1 & 0 \\
    0 & 0 & 1
    \end{smallmatrix}\Big) \right>,$
a group without maximal transvection root space, acting on $V=\FF_5^3$
with $e=\text{order}(\om)$. \cref{mainthm} implies
\begin{equation*}
    \begin{aligned}
        &\Hilb\Big(\big(
        \faktor{S}{\mathfrak{m}^{[5^m]}}\big)^{G}, \ t\Big)
        = \Big(\mfrac{1-t^{5^m}}
{1-t\ \ \ }\Big)\ 
\Big(
\mfrac{1-t^{5^m}}{1-t^5\ }\Big)\ 
\Big(\Big(
\mfrac{
1- t^{5^m-1}
}{1-t^e\ \ \ }\Big)
+
t^{5^m-1}
\Big( \mfrac{1-t^5}{1-t\ } \Big)
\Big) \ .
    \end{aligned}
\end{equation*}}
\end{example}

We take the limit as $t$ approaches $1$ in
\cref{mainthm} to obtain
the dimension:
\begin{cor}
\label{dim}
Suppose $G$ is a 
subgroup of $\Glp$ fixing a hyperplane. The dimension of $(S/\Ipm)^G$ as
an $\FF_p$-vector space is
$$
 \dim_{\, \FFp}\!
 \big(\faktor{S}{\Ipm}\big)^G 
 = 
 p^{m(n-1)}+p^{m(n-1)-\ell}\, 
 \Big(\mfrac{p^m-1}{e}\Big)\, .
$$
\end{cor}

\begin{remark}
{\em
Note that the Hilbert series in \cref{mainthm}
agrees with the series we expect
in the nonmodular case,
as then
all the reflections 
are semisimple
and $\ell=0$:
\begin{equation*}
\begin{aligned}
 \thilb\Big( \big(\faktor{S}{\Ipm}\big)^G,\ t \Big)
  \ =&\ 
   \mfrac{(1-t^{p^m})^{n-1}(1-t^{p^m+e-1})}{(1-t)^{n-1}(1-t^e)}\, .
\end{aligned}
\end{equation*}
The basic invariants
have degrees $1,\ldots, 1, e$
in this case and the series above
describes
$$
\big(\faktor{S}{\Ipm}\big)^G
\ = \
\faktor{\FFp[x_1, x_2, \dots, x_{n-1}, x_n^{e}]\, }{
(x_1^{p^m}, \ldots, x_n^{p^m})}\ .
$$
Compare with~\cite[Example~1.4]{LRS}.
}
\end{remark}

\begin{remark}{\em
When $G$ contains
no semisimple elements
(in the modular case),
$e=1$ and $\ell$
is just the minimum number of
generators of $G$.  
\cref{mainthm} implies
that
\begin{equation*}
    \begin{aligned}
        \Hilb\Big( \big(\faktor{S}{\Ipm}\big)^{G},\ t \Big) 
        &= \mfrac{(1-t^{p^m})^{n-1}(1-t^{p^m-1})}{(1-t^p)^{\ell}(1-t)^{n-\ell}}  +t^{p^m-1}
        \Big( \mfrac{1-t^{p^m}}{1-t} \Big)^{n-1}.
    \end{aligned}
\end{equation*}
}
\end{remark}

\section{Full Pointwise
Stabilizers over $\FF_q$
and Orbits}
\label{q}
The story is more complicated when generalizing to arbitrary finite fields $\FF_q$ for $q$ a prime power. 
The basic invariants for an
arbitrary subgroup of $\GLq$ fixing a hyperplane $H$ in $V=\FF_q^n$ pointwise can be described inductively (see ~\cite{HartmannShepler}). However, 
some of the previous ideas apply to
give the Hilbert series for full pointwise stabilizer subgroups.
Throughout this section,
fix a hyperplane $H$ in $V=\FFq^n$
and consider
the pointwise stabilizer
group
$G=\GLq_H=\{g\in \GLq: g|_H=1\}$.
In this case, $S^G$  is again a polynomial ring:
After change-of-basis, we may assume
$S^{G}=\FFq[\td {f}_1, \dots \td {f}_n]$
for
$$
\td f_1=x_1^q-x_1 x_n^{q-1},
\ldots,
\td f_{n-1} = x_{n-1}^q - x_{n-1} x_n^{q-1}
,\ 
\  
\td f_n=x_n^{q-1} \, . 
$$ 
Many results from
our previous sections 
hold
for these full stabilizer subgroups; 
for brevity, we highlight below only the more subtle
changes in the arguments.

Define polynomials in 
$\FF_q[
\td f_1, \ldots, \td f_n]$ for $1 \leq a \leq b < n$
by
\begin{equation*}
    \td h_0 = \td f_n^{\, 1+(q-1)^{-1}(q^m-1)},  \
    \td h_{1,a} =  \sum_{k=0}^{m-1}
    \td f_n^{\, 1+(q-1)^{-1}(q^m-q^{m-k})}\, 
    \td f_a^{\,q^{m-k-1}}
    ,\
    \td h_{2,a,b} 
    = \td f_a^{\, q^{m-1}}
    \td f_b^{\,q^{m-1}}.
\end{equation*}

The next observation is an analog of ~\cref{lemma:definition of invariants in frobenius ideal}; 
the proof
is straight-forward.
\begin{lemma} 
\label{frob invariants in ideal q}
For $G=\GLq_H$,
 $\td h_0, \, \td h_{1,a}, \, \td h_{2,a,b}$ lie in $S^{G}\cap\Iqm$
for $1 \leq a \leq b < n$.
\end{lemma}

The next result, analogous to \cref{lemma: leading monomial division},
uses a monomial ordering as in \cref{section: description of A}.

\begin{lemma}
\label{leading}
Say $G=\GLq_H$ and
$f$ in 
$S^{G}\cap \Iqm$ is homogeneous in the variables $\td f_1, \ldots, \td f_n$. Then $\LM_{S^G}(f)$ is divisible by $\td f_n$ or some $\td h_{2,a,b}$ with $1 \leq a \leq b <n$.
\end{lemma}

\begin{proof}
The
argument follows the proof of ~\cref{lemma: leading monomial division} with $p$ changed to $q$ and $e$ to $q-1$ until \cref{dInTermsOfa},
which, in our setting, becomes
\begin{equation}
\label{dintermsofaq}
d=
q^{m-1}+(q-1)a
\quad\text{and}\quad
i = 1+qa
\quad\text{for some}\quad
0 \leq a < \sum_{k=0}^{m-2}
q^k\,.
\end{equation}
We use the base $p$ expansions
of $a$ and $d$ to show instead that $\sum_{k=0}^{m-2}q^k \leq a$,
writing
$$
    a=
    \sum_{k=0}^{(m-2)r} 
    a_k\, p^k
    \quad\text{ and }\quad
 d = 
 \sum_{k=0}^{(m-1)r}
 d_k\, p^k
 \quad\text{
for some }\ \ 
   0 \leq a_k,\, d_k <p \,. 
$$   
We compare the base $p$
coefficients $d_k$ and $a_k$. 
Lucas' Theorem
implies that
\begin{equation*}
    0\neq 
    \mbinom{d}{i} = \mbinom{d_0}{1}\prod_{k=1}^{r-1}\mbinom{d_k}{0}\prod_{k=r}^{(m-1)r}
    \mbinom{d_{k}}
    {a_{k-r}}
\quad    \text{ as }
    i =  1+\sum_{k=r}^{(m-1)r}
    a_{k-r}\, p^{k}\, ;
\end{equation*}
since no factor in the product vanishes,
we conclude that $d_0 \geq 1$ and each $a_{k-r}\leq d_k$ for $r \leq k \leq (m-1)r$. 
 \cref{dintermsofaq} then provides direct
 comparison of $d_k$ and $a_k$,
\begin{equation}\label{d q equation}
 \sum_{k=0}^{(m-1)r}d_k\, p^k
 \ =\
    d\ =\ q^{m-1}-\sum_{k=0}^{r-1}a_k p^k +\sum_{k=r}^{(m-2)r}(a_{k-r}-a_k)p^{k}+\sum_{k=(m-2)r}^{(m-1)r}a_{k-r}p^k.
\end{equation}
We now regroup base $p$ as needed
and show inductively that
$0<a_0\leq a_r\leq\ldots\leq a_{(m-2)r}$. More generally, we will show $a_{k}\leq a_{k+r}$ for $0 \leq k \leq (m-3)r$ when $m>2$.

First consider  $a_0$. Since 
$1\leq d_0$, ~\cref{d q equation} implies that $d_0=p-a_0$ and  $a_0 \neq 0$.  
Thus $d_i=p-a_i-1$ for $1 \leq i \leq r-1$ since $d_i\geq 0$. 
For $m>2$, next observe that $a_0 \leq a_r$
since $a_0\leq d_r$
and~\cref{d q equation} implies that
$d_r=p+a_0-a_r-1$
for $a_0 \leq a_r$ whereas
$d_r=a_0-a_r-1$ for $a_r < a_0$.
Similarly, $a_1 \leq a_{r+1}$ since $a_1 \leq d_{r+1}$ and ~\cref{d q equation} implies that
$
    d_{r+1}= p+a_1-a_{r+1}-1$ for $a_1 \leq a_{r+1}$ whereas $d_{r+1}=a_1-a_{r+1}-1
    $ for $a_{r+1}<a_1$.
We iterate this argument and conclude
that $a_{k}\leq a_{k+r}$ for $0 \leq k \leq (m-3)r$. In particular,
$0<a_0\leq a_r\leq\ldots\leq a_{r(m-2)}$.
The result follows as in
the proof of ~\cref{lemma: leading monomial division}. \end{proof}

We adapt the proofs of 
Propositions~\ref{groebner basis}
and~\ref{lemma:hilbert series of invariant mod ideal} 
using Lemmas~\ref{frob invariants in ideal q} and \ref{leading}
to obtain 

\begin{prop}
\label{groebner basis q}
\label{hilb of A q}
For $G=\GLq_H$,
\begin{itemize}
\item
    $\{
    \td h_0,\, 
    \td h_{1,a},\,
    \td h_{2,a,b}:\,  1\leq a \leq b < n\}$
is a Groebner basis of $S^{G}\cap \Iqm$, 
\rule[-1.5ex]{0ex}{3ex}
\item
$
\Hilb\Big(
\faktor{S^{G} +\Iqm}{\Iqm},\ t\Big) 
=
\Big(
\mfrac{1-t^{q^m}}
{1-t^q}
\Big)^{n-1}\ 
\Big(\mfrac{
1-t^{q^m+q-2}+(n-1)\, t^{q^m}(1-t^{q-1})
}{1-t^{q-1}} \Big)
\, .
$
\end{itemize}
\end{prop}

The next result is an
 analog of ~\cref{lemma: description of initial term of invariants mod ideal in general case};
 we have only found
 case-by-case arguments
 for generalizing
 this key lemma
to arbitrary subgroups $G$ fixing a single hyperplane.

\begin{lemma}
\label{desc initial term q}
Say $G=\GLq_H$.
Suppose $f+\Iqm$ lies in $(S/\Iqm)^{G}$
with $f$ homogeneous in
$x_1, \ldots, x_n$.
Then $\LM_S(f)$
lies in
$$
\Iqm
\qquad{ or }\quad
\FF_q[x_1, \ldots, x_{n-1}, \td f_n^{\, (q^{m}-1)/(q-1)}]
\quad{ or }\quad
\FF_q[x_1^q, \ldots, x_{n-1}^q,
\td  f_n]\, .
$$
\end{lemma}

\begin{proof}
Say $M=\LM_S(f)$ does not lie in $\Iqm$ or in $\FF_q[x_1, \ldots, x_{n-1}, \td  f_n^{\, (q^{m}-1)/(q-1)}].$
Then
\begin{equation*}
    M=
   x_1^{b_1}\cdots  x_k^{b_k}
   \cdots x_{n-1}^{b_{n-1}}\,
   x_n^{b_n}
   \quad
   \text{ for some }
      \ 
   b_1, \ldots, b_{n-1}
   <q^m 
   \text{ and }
   b_n< q^m-1\, .
\end{equation*}

Invariance under 
the generator
$g_n$ of the semisimple elements in $G$
implies that $b_n$ is divisible by $q-1$ as in the proof of
\cref{lemma: description of initial term of invariants mod ideal in general case}
and so $b_n\leq q^m-q$.

Say $q$ does not divide some $b_k$ 
with $k<n$. Let $g$
in $\Glq_H$ be
the element mapping $x_k$ to $x_k+x_n$ and fixing $x_i$ for $i\neq k$ as in ~\cref{section: description of A}.
Arguments as in the proof of \cref{lemma: description of initial term of invariants mod ideal in general case}
show that $p$ divides $b_k$ and thus 
\begin{equation*}
    b_k=p^tc_k \quad \textrm{ for } t<r=\dim_{\FFp}(\FF_q) \textrm{ and } \gcd(p,c_k)=1.
\end{equation*}

We argue that the monomial
$
    x_\beta = M \cdot (x_n/x_k)^{p^t} 
$
lies outside $\Iqm$
but appears with  nonzero
coefficient in $gf-f$, contradicting the fact that $gf-f \in \Iqm$. 
The leading term
of $gM-M$
in $\FF_q[x_1, \ldots, x_n]$
is $c_k x_\beta\neq 0$ as 
$\binom{p^t c_k}{p^t}
=\binom{c_k}{1}$ by Lucas' Theorem.
We claim that for any
other
monomial $N$ of $f$, the coefficient of $x_\beta$ in $gN-N$ is zero.
Since $g$ fixes $x_i$
with $i\neq k$, we 
may assume
$\deg_{x_i}(N)
=\deg_{x_i}(M)$ for $i\neq k$, 
$i<n$.
Since $\LM_S(f)=M$,  $\deg_{x_k}(N) < b_k$ and $\deg_{x_n}(N)>b_n$. But $q-1$ divides $\deg_{x_n}(N)$
as $f$ is invariant under the diagonal reflection sending
$x_n$ to $x_n^{q-1}$
and thus 
$
\deg_{x_n}(N)\geq b_n + q - 1$ and $\deg_{x_k}(gN-N)\leq
\deg_{x_k}(N)\leq b_k -q+1< b_k - p^t
=\deg_{x_k}x_\beta$,
and the coefficient of $x_\beta$ in $gN-N$ is zero.
Note that $gf-f \notin \Iqm$ since
$b_n+p^t < q^m$ as  $b_n \leq q^m-q$.
\end{proof}

We obtain a direct sum decomposition analogous to \cref{corollary:description of invariant space}.
\begin{prop}\label{directsum q}
For $G=\GLq_H$,
the invariants are $\big(\faktor{S}{\Iqm}\big)^{G} 
= 
A_{G} \oplus B_{G}\, 
$ 
for 
$$
\begin{aligned}
A_{G}&=
(S^{G}+\Iqm)
/\Iqm
\qquad\text{ and }\\
    B_{G}&=\, 
\mathbb{F}_q[
\tilde f_1, \dots , 
\tilde f_{n-1}]
\text{-span} \Big\{ x_1^{a_1}\dots 
x_{n-1}^{a_{n-1}}x_n^{q^m-1}+\Iqm, 
\text{ for }
0 \leq a_i < q,\sum_{i=1}^{n-1}a_i \geq 2 
\Big\}
.
\end{aligned}
$$
\end{prop}
\begin{proof}
One may easily adapt the proofs of \cref{proposition:A+B}
and 
\cref{corollary:description of invariant space}
to the case of $\FF_q$
using
\cref{groebner basis q}
and
\cref{desc initial term q}.
\end{proof}

The proof of ~\cref{lemma:hilbert series of B} can be modified to give the Hilbert series of $B_{G}$:

\begin{lemma}\label{hilb of B q}
For any hyperplane $H$ in $V=\FFq^n$
and $G=\GLq_H$,
\[\Hilb(B_{G},\ t) 
\ =\ 
t^{q^m-1}
\Big(\Big( \mfrac{1-t^q}{1-t\ } \Big)^{n-1} - (n-1) t - 1\Big)
\Big(\mfrac{1-t^{q^m}}{1-t^q\ } \Big)^{n-1} \, .
\]
\end{lemma}

Finally, Propositions~\ref{hilb of A q} and \ref{directsum q} and \cref{hilb of B q}
give \cref{introthm}
of the Introduction:

\begin{cor}
For any hyperplane $H$ in $V=\FFq^n$
and $G=\GLq_H$,
\label{MainThm q}
\begin{equation*}
    \begin{aligned}   \Hilb\Big(\big(\faktor{S}{\Iqm}\big)^{G},\ t \Big)
         =& \
    \Big(\mfrac{ 1-t^{q^m}}{1-t^q\ \ }\Big)^{n-1}
    \Big(\Big(
    \mfrac{ 1-t^{q^m-1}}{1-t^{q-1}\ \ \ \ }\Big)
    +t^{q^m-1}
    \Big(
    \mfrac{1-t^q}{1-t\ } \Big)^{n-1}\Big) \\
      =& \,
     \thilb(S^{G}, \, t)(1-t^{q^m})^{n-1}
        \Big( 1 - t^{q^m-1}\!
        +(1-t^{q-1})t^{q^m-1}
        \Big(        \mfrac{1-t^q}{1-t\ }
        \Big)^{\! n-1}\Big)
\\
     =& \
        [q^{m-1}]_{t^q}^{n-1}
        \qbin{m}{1}{{q,t}}
        + 
        t^{q^m -1}
        [q^m]_{t}^{n-1}
        \qbin{m}{0}{{q,t}}
\, .
    \end{aligned}
\end{equation*}
\end{cor}

We take the limit $t\mapsto 1$ for the following corollary. 
\begin{cor}
\label{dimq}
For any hyperplane $H$ in $V=\FFq^n$
and $G=\GLq_H$,
$$
 \dim_{\, \FFq}\!
 \big(\faktor{S}{\Iqm}\big)^{G} 
 = 
 q^{m(n-1)}+q^{(m-1)(n-1)}\, 
 \Big(\mfrac{q^m-1}{q-1}\Big)\, .
$$
\end{cor}

\subsection*{Orbits
and the dimension
of the invariant space}\label{orbits section}
The conjecture of Lewis, Reiner, and Stanton
~\cite{LRS}
giving the Hilbert series
for the $\Glq$-invariants
in
$S/\Iqm$
specializes to 
a conjecture
for the dimension of
the invariants
as an $\FF_q$-vector space.
They show
this specialization
gives the number of orbits
for $\Glq$
acting on 
the vector
space $V'=(\FF_{q^m})^n$,
see~\cite[Section~7.1 and
Theorem~6.16]{LRS}.

\cref{dim}
gives the dimension of
the
$G$-invariants in
$S/\Ipm$
over $\FF_p$ 
for any group $G$
fixing a hyperplane.
Below we prove that
this integer gives
the number of orbits
for 
$G$ as a subgroup of $\Gl_n(\FF_{p})$ 
acting 
on the vector 
space $V'=(\FF_{p^m})^n$
(with canonical coordinate-wise action
induced from the embedding
$\FF_p \subset \FF_{p^m}$).
This result
thus proves a special
case of the 
conjecture of
Lewis, Staton, and Reiner.
Here again,
$\ell=\dim_{\FFp}(\text{RootSpace}(G))\cap H$ with
$e$ the maximal order of a semisimple element in $G$. 
\begin{cor}
Suppose $G\leq \Glp$ is a reflection group fixing a hyperplane $H$ in $V=\mathbb{F}_{p}^n$.
The number of orbits of points in 
$V'=(\FF_{p^m})^n$ 
under the action of $G$ 
is equal to the dimension
over $\FF_p$
of the $G$-invariants
in $S/\Ipm$:
\begin{equation*}
    \dim_{\FFp}\big(\faktor{S}{\Ipm}\big)^G\ =\ 
     p^{m(n-1)}+p^{m(n-1)-\ell}\ 
 \Big(\mfrac{p^m-1}{e}\Big)\,
 \ = \ 
     \text{ \# orbits
    of $G$ on } (\FF_{p^m})^n\,
    .
\end{equation*}
\end{cor}
\begin{proof}
\cref{dim} records the dimension;
we count orbits here.
Let $H'$ be the image of $H$
under the coordinate-wise
embedding $V\hookrightarrow V'$. 
Choose a basis
$x_1, \ldots, x_n$
of $(V')^*$
dual to the standard coordinate
basis as in \cref{choosebasis}
with
$H'=\ker x_n$ in $V'$.
The number of points with orbit size $1$ is the number of points on the hyperplane $H'$, namely, 
$(p^m)^{n-1}$. 
Two points $v$ and $u$ lying 
in the complement $(H')^c$
of $H'$ in $V'$ 
lie in the same $G$-orbit
if and only if
$x_i(v)=x_i(u)$ for 
$i \leq \ell$ and 
$x_n(u)$ lies in $
\FFp x_1(v)+ \dots + \FFp x_{\ell}(v)+\langle \om \rangle x_n(v)$
for $\om$ a primitive $e$-th root-of-unity in $\FFp$.
Thus a fixed $v$ in $(H')^c$ 
has orbit size $p^{\ell} e $
whereas $|(H')^c|
=(p^m)^{n-1}(p^m-1)$
and
$$
\text{\# of orbits in }
(H')^c
=
\mfrac{|(H')^c|}
{\text{ size of an orbit in }
(H')^c}
=
p^{m(n-1)-\ell}\
\big(\mfrac{p^m-1}{e}\big)
\, . 
$$
The total number of orbits for $\Glp_H$ acting on $\mathbb{F}_{p^m}$ is then 
$$
\begin{aligned} 
\text{\# of orbits} 
\ &=\ 
(\text{\# orbits on $H'$})
+
(\text{\# orbits on $(H')^c$})
\\
\ &=\
p^{m(n-1)} 
+
p^{m(n-1)-\ell}\
\big(\mfrac{p^m-1}{e}\big)
\,  .
\end{aligned}
\vspace{-2ex}
{}_{}
$$
\end{proof}

A similar proof 
with
$q$ for $p$ and $q-1$
for $e$
using \cref{dimq}
gives
\cref{introcor}
from the Introduction;
we suspect a similar
statement holds
for any reflection
group:
\begin{cor}
\label{counting orbitsq}
For any hyperplane $H$ in $V=\FFq^n$,
the number of orbits in 
$(\FF_{q^m})^n$ 
under the action of $G=\GLq_H$ 
is 
\begin{equation*}
    \dim_{\FFq}\big(\faktor{S}{\Iqm}\big)^{G} 
    = \,
     q^{m(n-1)}+q^{(m-1)(n-1)}\ 
 \Big(\mfrac{q^m-1}{q-1}\Big)\,
  =         
  q^{m(n-1)}{m \brack 0}_q
                + 
        q^{(m-1)(n-1)}{ m \brack 1}_q 
        \, .  
\end{equation*}
\end{cor}


\section{Lewis, Reiner, and Stanton conjecture}\label{section: LRS tie in}
We use our results in previous sections to bound the exponents of
$x_1, \ldots, x_n$ 
in any invariant of $S/\Iqm$
under the full general linear group
$\GLq$
for a prime power $q$.

\begin{prop}
Say  $f+\Iqm \in 
(S/ \Iqm)^{\Glq}$.
For any monomial $M\notin \Iqm$ 
in $x_1, \ldots, x_n$
of $f$,
either
$M=x_1^{q^m-1} x_2^{q^m-1}\cdots
x_n^{q^m-1}$ or
$\deg_{x_i}(M)\leq q^m-q$
for all $i$.
\rule{0ex}{3ex}
\end{prop}

\begin{proof} 
We may assume $f$ is homogeneous
in $x_1, \ldots, x_n$
with no monomials lying in $\Iqm$.
By
\cref{desc initial term q}
with hyperplane $H=\ker x_n$
and ordering 
$x_1 > \dots > x_n$,
\begin{equation*}
    \LM(f)\in \FFq[x_1, \dots, x_{n-1}, x_n^{q^m-1}] \quad \textrm{ or } \quad \LM(f) \in \FFq[x_1^q, \dots, x_{n-1}^q, x_n^{q-1}] \, .
\end{equation*}
First suppose $\deg_{x_n}(\LM(f))= q^m-1$.
The element
$f+\Iqm$, 
and hence $f$ itself, is invariant under
the action of the symmetric group $\mathfrak{S}_n$
permuting the variables
as a subgroup of $\Glq$.
This forces
$\LM(f)=x_1^{q^m-1}
\cdots x_n^{q^m-1}=f$, as $f$ 
is homogeneous. 

Now assume 
$\deg_{x_n}(\LM(f))\neq q^m-1$,
so that $q$ divides $\deg_{x_1}(\LM(f))$. Since $f$ is invariant under the diagonal reflection with
$x_1\mapsto \om x_1$ for $\om$
a primitive
$(q-1)$-th root-of-unity,  
$(q-1)$ also divides $\deg_{x_1}
(\LM(f))$. Therefore, $q(q-1)$ divides $\deg_{x_1}(\LM(f))$
and
$\deg_{x_1}(\LM(f))\leq q^m-q$. 
Then
$\deg_{x_1}(M) \leq q^m-q$ for any monomial $M$ of $f$. 
As $f$ is 
$\mathfrak{S}_n$-invariant,
$\deg_{x_i}(M) \leq q^m-q$ for all $i$ as well.
\end{proof}

The previous proposition gives a bound on coefficients 
of the Hilbert series. 
Let ${\rm HF}$ be the {\em Hilbert function},
${\rm HF}(M,i)=\dim_{\FF} M_i$,
for any $\ZZ$-graded 
vector space
$M=\bigoplus M_i$.

\begin{cor}
We give a bound on the Hilbert
function
of $\Glq$-invariants:
$$
\begin{aligned}
 &
 {\rm HF}\Big(\big(\faktor{S}
 {\Iqm}\big)^{\Glq}
, n(q^m-1)\Big)=1 \quad \text{and}\\
&{\rm HF}\Big(\big(\faktor{S}
{\Iqm}\big)^{\Glq}
,\, i\Big)
\leq 
{\rm HF}\Big(
\faktor{S}{
(x_1^{q^m-q+1} ,
\ldots,x_n^{q^m-q+1})}
,i\Big) \quad \text{for } i \neq n(q^m-1)\, .
\end{aligned}
$$
\end{cor}


\section{Two dimensional vector spaces}\label{section: 2dim}

We now consider the $2$-dimensional
case and take a group $G$ of $\textrm{GL}_2(\Fp)$ fixing a hyperplane (line) of 
$V=(\mathbb{F}_p)^2$ pointwise.
Here, $\mathfrak{m}^{[p^m]}:=(x_1^{p^m},\, x_2^{p^m})$. 
We give a resolution of $S^G \cap \Ipm$ directly using syzygies, providing an alternate direct computation for the Hilbert series of $A_G = (S^G +\Ipm)/ \Ipm$.

\subsection*{Nonmodular Setting}
If $G$ contains no transvections, then $S^G \cap \Ipm$ is generated by $h=f_1^{p^m}$ and $h'=f_2^{1+e^{-1}(p^m-1)}$
and we obtain
an easy resolution for
$S^G \cap \Ipm$,
\begin{equation*}
    0  \longrightarrow F_1
    \xrightarrow{\ [\tau] \ } F_0 \xrightarrow{\ [h\ h']\ }  S^G \cap \mathfrak{m}^{[p^m]}
    \longrightarrow 0
    \, ,
\end{equation*}
where 
$F_1=S^G[-(2p^m+e-1)]$
and $F_0=S^G[-p^m]\oplus S^G[-(p^m+e-1)]$
with relation
$\tau=f_2^{1+e^{-1}(p^m-1)}h 
- f_1^{p^m}h'$. 
This gives Hilbert series 
\begin{equation*}
    \Hilb(S^G \cap \Ipm, \ t) = \mfrac{t^{p^m}+t^{p^m+e-1}-t^{2p^m+e-1}}{(1-t^e)(1-t)}
    =
    \Hilb(S^G, \ t)(t^{p^m}+t^{p^m+e-1}-t^{2p^m+e-1}) \ .
\end{equation*}

\subsection*{Modular setting}
Suppose now that $G$ contains 
a transvection. After conjugation,
$G=\langle 
(\begin{smallmatrix}
1 & 0\\
0 & \omega \\
 \end{smallmatrix}),
 (\begin{smallmatrix}
    1 & 1\\
    0 & 1\\
    \end{smallmatrix})
\rangle$
for some root-of-unity $\omega \in \mathbb{F}_p$ of order $e\geq 1$.
Here, 
$$
S^G=\FF_p[x_1, x_2]^G= \FF_p[f_1, f_2]
\quad\text{ for } f_1=x_1^p-x_1x_2^{p-1}
\text{ and } f_2=x_2^e \, .
$$
The Groebner basis 
$$ 
h_0=f_2^{1+e^{-1}(p^m-1)}
,\quad
h_1= \sum_{k=0}^{m-1}f_2^{1+e^{-1}(p^m-p^{m-k})}\ f_1^{p^{m-k-1}}
, \quad
h_2= f_1^{2p^{m-1}}
$$
(see \cref{definition:invariants})
of the ideal $S^G \cap \mathfrak{m}^{[p^m]}$
in the polynomial ring $S^G$
is small enough to directly provide a manageable
resolution of $S^G/ S^G \cap \mathfrak{m}^{[p^m]}$,
which we record below.

\begin{prop}\label{theorem:resolution with diagonalizable}
For $G$ a subgroup of $\GL_2(\FFp)$ containing a transvection, a graded free resolution of the $S^G$-module
$S^G \cap \mathfrak{m}^{[p^m]}$
is
\begin{equation*}
    0 \longrightarrow F_1 
    \xrightarrow{\ [\tau_{0,1}\ \tau_{1,2}] \ } F_0 \xrightarrow{\ [h_0\ h_1\ h_2]\ }  S^G \cap \mathfrak{m}^{[p^m]}
    \longrightarrow 0
\end{equation*}
for 
$$
\begin{aligned}
F_0=&\ S^G\big[-(p^m+e-1)\big] \oplus S^G\big[-(p^m+e)\big] \oplus S^G[-2p^m],
\quad\text{ and}
\\
F_1=&\ S^G\big[-(2p^m+e)\big] \oplus S^G\big[-(2p^m+e-1)\big]\, .
\end{aligned}
$$ 
\end{prop}
\begin{proof}
Buchberger's algorithm gives
generators for the first syzygy-module  in $(S^G)^3$ for
$S^G\cap \Ipm=(h_0, h_1, h_2)$,
namely,
\begin{equation*}
    \begin{aligned}
        \tau_{0,1}&= 
        (- f_1^{p^{m-1}} - \sum_{k=1}^{m-1} f_1^{p^{m-k-1}}f_2^{e^{-1}(p^m-p^k)},\ 
        f_2^{e^{-1}(p^m-1)},\ 0)
       \\
         \tau_{0,2}
         &= 
         (f_1^{2p^{m-1}},\ 0,\  - f_2^{  1+e^{-1}(p^m-1)}), 
         \quad\quad\textrm{ and}\\
        \tau_{1,2} 
         &= 
         (-\! \sum_{j,k=1}^{m-1}  f_1^{p^{m-j-1}\! + p^{m-k-1}}f_2^{e^{-1}(p^m-p^k-p^j+1)},\ 
         - f_1^{p^{m-1}} + \sum_{k=1}^{m-1}  f_1^{p^{m-k-1}}f_2^{e^{-1}(p^m-p^k)},\
         f_2)\, .
\end{aligned}
\end{equation*}
But $\tau_{0,2}$ is
redundant as
\begin{equation*}
    \begin{aligned}
        \tau_{0,2} &= \Big(\sum_{k=1}^{m-1}f_1^{p^{m-k-1}}f_2^{e^{-1}(p^m-p^k)}\tau_{0,1} - f_1^{p^{m-1}}
        \Big)\tau_{0,1} -   f_2^{e^{-1}(p^m-p^k)}\tau_{1,2}
        \, , 
    \end{aligned}
\end{equation*}
and the first syzygy-module
is generated over $S^G$
by just $\tau_{1,2}$  and $\tau_{0,1}$.  As these are 
linearly independent over $S^G$,
the second syzygy-module is trivial, and 
the result  follows.
\end{proof}

This 
gives an easy proof
of \cref{lemma:hilbert series of invariant mod ideal}
in the modular $2$-dimensional setting:
\begin{cor}
For $G$ a subgroup of 
$\GL_2(\FF_p)$ fixing a hyperplane
in $V=(\FFp)^2$ and containing 
a transvection,
\begin{equation*}    
 \Hilb\Big(
\faktor{ (S^G+\Ipm)}
{\Ipm},\ t \Big)    
    =
\Hilb(S^G,\ t)
(1-t^{p^m})
(1+t^{p^m}-t^{p^m+e-1}-t^{p^m+e})   
\, .
\end{equation*}
\end{cor}

\begin{proof}
By
~\cref{theorem:resolution with diagonalizable},
the Hilbert series for
$S^G\cap \Ipm$
is just the series for $F_1$
subtracted from
that for $F_0$.
The proposition then follows
from using the 
exact sequence
\begin{equation*}
    0 \longrightarrow S^G \cap \Ipm \longrightarrow S^G \longrightarrow \faktor{S^G}{ (S^G \cap \Ipm)}
    \cong \faktor{(S^G +\Ipm)}{ \Ipm}
    \longrightarrow 0.
\end{equation*}
\end{proof}




\end{document}